\documentclass[12pt]{article}

\usepackage[utf8]{inputenc}

\usepackage{amsmath,amsfonts,amssymb,amsthm}
\usepackage{xspace}
\usepackage{enumerate}
\usepackage{url}

\usepackage{etoolbox,fp}
 \usepackage{xcolor}
\usepackage{tikz}
\usetikzlibrary{fit}
\usetikzlibrary{backgrounds}
\usetikzlibrary{graphs}
\usetikzlibrary{arrows}
\tikzset{
  every picture/.style={>=stealth'},
  semithick,
  nolabel/.style={label/.code={}},
  graphs/declare={gap}{__{}[draw=none,fill=none] },
  graphs/declare={gp2}{__{}[draw=none,fill=none]  -!-  __{}
    [draw=none,fill=none] },
  graphs/declare={gp3}{__{}[draw=none,fill=none]  -!-  __{}
    [draw=none,fill=none] -!-  __{}[draw=none,fill=none] },
  smallgraphs/.style={%
    graphs/every graph/.style={
      branch right=4mm,
      grow down=4mm,
      empty nodes,
      edges={line width=0.2pt},
      nodes={circle,inner sep=1.2pt}
    },
    lab/.style={draw=none,fill=none,rectangle}
  },
  every graph/.style={branch right=5mm,grow down=5mm}
}

\makeatletter
\newcommand{\mathnamenodebasic}[2][]{%
  \node[circle,fill,inner sep=2pt,#1,
  ] (#2) {};%
  \@ifundefined{mnn@savelabel}{%
    \xdef\mnn@savelabel{#2/#1}%
  }{%
    \xdef\mnn@savelabel{\mnn@savelabel,#2/#1}%
  }%
}
\newcommand{\matrixgraphfinal}[1][]{%
  \foreach \lab/\opts in \mnn@savelabel {
    \scoped[label distance=1.2mm,#1,\opts]%
    \node[\opts,rectangle,fill=none,draw=none,%
      label={[rectangle]{$%
      \pgfkeysvalueof{/mnn/label prefix}%
      \lab
      \pgfkeysvalueof{/mnn/label suffix}%
      $}},at={(\lab)}] {};%
  };%
  \global\let\mnn@savelabel\relax
}
\newcommand{\mnn@exec}[1][]{\mathnamenodebasic[#1]{\mnn@name}}
\def\mnn@gathername#1{%
  \edef\mnn@name{\mnn@name#1}\mnn@work%
}

\def\mnn@work{%
  \@ifnextchar\pgfmatrixnextcell{%
    \mnn@exec
  }{%
  \@ifnextchar\pgfmatrixendrow{%
      \mnn@exec
  }{%
      \@ifnextchar[{%
      \mnn@exec
      }{%
      \mnn@gathername%
      }%
    }%
  }%
}
\newcommand{\mathnamenode}{\def\mnn@name{}\mnn@work}
\makeatother

\makeatletter
\newcommand{\matrixgraph}{
  \begingroup
  \catcode`\&=13%
  \matrixgraph@work
}
\newcommand{\matrixgraph@work}[3][]{
  \matrix[fill=none,draw=none,rectangle,inner sep=0pt,%
    matrix anchor=north west,#1,%
    execute at begin cell=\mathnamenode] {#2};%
  \endgroup%
  \graph[use existing nodes,%
    default edge operator=complete bipartite]{#3};%
  \matrixgraphfinal[#1]%
}
\makeatother

\newcommand{\colclass}[2][0]{
\makeatletter
\xdef\clcl@verts{}
\foreach \v in {#2}{\xdef\clcl@verts{\clcl@verts(\v)}}
\begin{scope}[on background layer]
\node[fill=black!20,rotate fit=#1,fit/.expand once={\clcl@verts},inner sep=4pt,rectangle,rounded corners=8pt,draw=none] {};
\end{scope}
\makeatother
}

\newcommand{\ra}{15mm}

\theoremstyle{plain}
\newtheorem{theorem}{Theorem}[section]
\newtheorem{lemma}[theorem]{Lemma}
\newtheorem{proposition}[theorem]{Proposition}

\newtheorem{corollary}[theorem]{Corollary}
\theoremstyle{definition} 

\newtheorem{remark}[theorem]{Remark}
\newtheorem{definition}[theorem]{Definition}
\newtheorem{assumption}[theorem]{Assumption}

\newcounter{claim}
\renewcommand{\theclaim}{\Alph{claim}}
\newenvironment{claim}{\refstepcounter{claim}%
\par\medskip\par\noindent{\it Claim~\theclaim.~}~\rm}%
{\par\smallskip\par}
\newenvironment{subproof}{\par\noindent{\sl Proof of Claim~\theclaim.~}}%
{$\,\triangleleft$\par\medskip\par}

\makeatletter
\def\@gifnextchar#1#2#3{\let\@tempe#1\def\@tempa{#2}\def\@tempb{#3}%
  \futurelet\@tempc\@gifnch}

\def\@gifnch{\ifx\@tempc\@sptoken\let\@tempd\@tempb%
  \else\ifx\@tempc\@tempe\let\@tempd\@tempa\else\let\@tempd\@tempb\fi\fi\@tempd}

\def\SK@set#1{\left\{#1\right\}}
\def\SK@@set#1#2{\{#1\,:\,
    \begin{array}{@{}l@{}}#2\end{array}
\}}
\def\SK@mset#1{\left\{\!\!\left\{#1\right\}\!\!\right\}}
\def\SK@@mset#1#2{\{\!\!\{#1\,:\,
    \begin{array}{@{}l@{}}#2\end{array}
\}\!\!\}}
\def\BIG@set#1{\Big\{#1\Big\}}
\def\BIG@@set#1#2{\Big\{#1\:\Big|\:
    \begin{array}{@{}l@{}}#2\end{array}
\Big\}}
\newcommand{\Set}[1]{\@gifnextchar\bgroup{\SK@@set{#1}}{\SK@set{#1}}}
\newcommand{\Mset}[1]{\@gifnextchar\bgroup{\SK@@mset{#1}}{\SK@mset{#1}}}
\newcommand{\Bigset}[1]{\@gifnextchar\bgroup{\BIG@@set{#1}}{\BIG@set{#1}}}
\makeatother

\newcommand{\refeq}[1]{(\ref{eq:#1})}

\newcommand{\function}[2]{:#1 \rightarrow #2}
\newcommand{\compl}[1]{\overline{#1}}

\DeclareMathOperator{\aut}{Aut}
\DeclareMathOperator{\inv}{Inv}
\DeclareMathOperator{\cay}{Cay}
\DeclareMathOperator{\diam}{diam}

\newcommand{\WL}[1]{\ensuremath{#1\text{-}\mathrm{WL}}\xspace}
\newcommand{\wl}{\WL 2}
\newcommand{\kwl}{\WL k}
\newcommand{\alg}[3]{\mathrm{WL}_{#1}^{#2}(#3)}
\newcommand{\algt}[1]{\mathrm{WL}_2(#1)}
\newcommand{\kalgt}[1]{\mathrm{WL}_k(#1)}
\newcommand{\wlh}[2]{\mathrm{WL}_{#1,#2}}
\newcommand{\wlhi}[2]{\mathrm{WL}(#1,#2)}
\newcommand{\uhi}[1]{\mathrm{U}(#1)}
\newcommand{\tri}[1]{\mathrm{R}(#1)}
\newcommand{\vtclass}{\mathrm{VT}}

\newcommand{\baru}{{\bar u}}
\newcommand{\barx}{{\bar x}}
\newcommand{\bary}{{\bar y}}

\newcommand{\bF}{\mathbb{F}}
\newcommand{\bZ}{\mathbb{Z}}

\newcommand{\cP}{\mathcal{P}}
\newcommand{\cS}{\mathcal{S}}
\newcommand{\cT}{\mathcal{T}}
\newcommand{\cX}{\mathcal{X}}

\newcommand{\fp}{\bF_p}
\newcommand{\fpa}{\fp^+}
\newcommand{\fpm}{\fp^\times}

\newcommand{\outn}[2]{N_{#2}(#1)}

\newcommand{\clo}[1]{\mathit{Cl}(#1)}

\newcommand{\symdiff}{\bigtriangleup}

\title{The Weisfeiler-Leman Algorithm\\ and Recognition of Graph Properties}

\author{%
Frank Fuhlbrück\thanks{Institut für Informatik, Humboldt-Universität zu Berlin, Germany.}, 
Johannes Köbler${}^*$, 
Ilia Ponomarenko\thanks{School of Mathematics and Statistics of
  Central China Normal University, Wuhan, China and
  Steklov Institute of Mathematics at St. Petersburg, Russia.}, 
Oleg Verbitsky${}^*$\,\thanks{Supported by DFG grant
KO 1053/8--1. On leave from the IAPMM, Lviv, Ukraine.}}

\date{}

\begin{document} 

\maketitle

\begin{abstract}
  The $k$-dimensional Weisfeiler-Leman algorithm (\kwl) is a very
  useful combinatorial tool in graph isomorphism testing. We address the applicability
  of \kwl to recognition of graph properties. Let $G$ be an input graph
  with $n$ vertices. We show that, if $n$ is prime, then vertex-transitivity
  of $G$ can be seen in a straightforward way from the output of \wl
  on $G$ and on the vertex-individualized copies of $G$. However, if
  $n$ is divisible by 16, then \kwl is unable to distinguish between
  vertex-transitive and non-vertex-transitive graphs with $n$ vertices
  as long as $k=o(\sqrt n)$. Similar results are obtained for recognition
  of arc-transitivity.
\end{abstract}

\section{Introduction}\label{s:intro}

The $k$-dimensional Weisfeiler-Leman algorithm (\kwl),
whose original, 2-dimensional version \cite{WLe68} appeared in 1968,
has played a prominent role in isomorphism testing already
for a half century. Given a graph $G$ with vertex set $V$,
\kwl computes a canonical coloring $\kalgt G$ of the Cartesian power $V^k$.
Let $\Mset{\kalgt G}$ denote the multiset of colors appearing in $\kalgt G$.
The algorithm decides that two graphs $G$ and $H$ are isomorphic
if $\Mset{\kalgt G}=\Mset{\kalgt H}$, and that they are non-isomorphic
otherwise. While a negative decision is always correct,
Cai, Fürer, and Immerman \cite{CaiFI92}
constructed examples of non-isomorphic graphs $G$ and $H$ with $n$ vertices
such that $\Mset{\kalgt G}=\Mset{\kalgt H}$ as long as $k=o(n)$.
Nevertheless, a constant dimension $k$ suffices to correctly decide
isomorphism for many special classes of graphs (when $G$ is in the
class under consideration and $H$ is arbitrary). For example,
$k=2$ is enough if $G$ is an interval graph \cite{EvdokimovPT00},
$k=3$ is enough for planar graphs \cite{KieferPS19}, and there is a constant $k=k(M)$
sufficient for all graphs not containing a given graph $M$ as a minor \cite{Grohe12}.
Last but not least, \kwl is an important component in Babai's
quasipolynomial-time algorithm \cite{Babai16} for general graph isomorphism.

In the present paper, we initiate a discussion of the applicability
of \kwl to recognition of graph properties rather than to testing isomorphism.
That is, given a single graph $G$ as input, we are interested in knowing
which properties of $G$ can be detected by looking at $\kalgt G$.
Some regularity properties are recognized in a trivial way.
For example, $G$ is strongly regular if and only if \wl
splits $V^2$ just in the diagonal $\Set{(u,u)}{u\in V}$,
the adjacency relation of $G$, and the complement. It is worth
of mentioning that isomorphism of strongly regular graphs is
considered to be a hard problem for \kwl.
It is only known \cite{Babai80} that \kwl correctly decides
isomorphism of strongly regular graphs with $n$ vertices
if $k\ge 2\sqrt n\log n$ (see, e.g., also~\cite{BabaiCSTW13}).

For a property $\cP$, we use the same character $\cP$ to denote
also the class of all graphs possessing this property.
Suppose that the isomorphism problem for graphs in $\cP$ is known to be
solvable by \kwl for some $k$. This means that, for every $G\in\cP$,
$\Mset{\kalgt G}$ is a complete isomorphism invariant of $G$
and hence, at least implicitly, $\kalgt G$ contains the information
about all properties of $G$. It is, however, a subtle question whether
any certificate of the membership of $G$ in $\cP$ can be extracted
from $\kalgt G$ efficiently. We can only be sure that \kwl
distinguishes $\cP$ from its complement, in the following sense:
If $G\in\cP$ and $H\notin\cP$, then $\Mset{\kalgt G}\ne\Mset{\kalgt H}$.
However, given the last inequality, we will never know whether
$G\in\cP$ and $H\notin\cP$ or whether $H\in\cP$ and $G\notin\cP$.
As a particular example, the knowledge that \wl decides isomorphism
of interval graphs or that \WL3 decides isomorphism of planar graphs
does not seem to imply, on its own, any recognition algorithm
for these classes.

We address the applicability of \kwl to recognition of properties
saying that a graph is highly symmetric.

\paragraph{Vertex-transitivity.}
A graph $G$ is \emph{vertex-transitive} if every vertex can be taken
to any other vertex by an automorphism of $G$. It is unknown whether
the class of vertex-transitive graphs is recognizable in polynomial
time. The isomorphism problem for vertex-transitive graphs reduces to
their recognition problem, and its complexity status is also open.
In the case of graphs with a prime number $p$ of vertices, a polynomial-time
recognition algorithm is known due to Muzychuk and Tinhofer \cite{MuzychukT98}.
Their algorithm uses \wl as preprocessing and then involves
a series of algebraic-combinatorial operations
to find a Cayley presentation of the input graph.
Our first result, Theorem \ref{thm:p}, shows a very simple,
purely combinatorial way to recognize vertex-transitivity of
a graph $G$ with $p$ vertices. Indeed, vertex-transitivity can immediately be
detected by looking at the outdegrees of the monochromatic digraphs
in $\algt G$ and $\algt{G_u}$ for all copies of $G$ with an
individualized vertex $u$. Our algorithm takes time $O(p^4\log p)$,
which is somewhat better than the running time $O(p^5\log^2 p)$
of the algorithm presented in~\cite{MuzychukT98}.

If $p$ is prime, then the vertex-transitive graphs with $p$ vertices
are known to be \emph{circulant graphs} \cite{Turner67}, i.e., Cayley graphs of
the cyclic group of order $p$. 
The research on circulant graphs has a long history; see, e.g., \cite{Babai77,MuzychukKP99}.
This class of graphs can be recognized in polynomial time \cite{EvdokimovP04},
but whether or not this can be done by means of \kwl is widely open.
The dimension $k=2$ would clearly suffice if the algorithm could identify
a cyclic order of the vertices in an input graph corresponding to its automorphism.
However, this is not the case because such an order is, in general, not unique
and is not preserved by automorphisms, even when the number of vertices is prime.

The analysis of our algorithm is based on the theory of coherent configurations
(see a digest of main concepts in Section \ref{ss:ccs}).
In fact, our exposition, apart from the well-known facts
on circulants of prime order, uses only several results
about the \emph{schurity property} of certain coherent configurations.

Furthermore, we explore the limitations of the \kwl-based
combinatorial approach to vertex-transitivity by proving that, if
$n$ is divisible by 16, then \kwl is unable to distinguish between
vertex-transitive and non-vertex-transitive graphs with $n$ vertices
as long as $k=o(\sqrt n)$ (see Theorem \ref{thm:vt}).
This excludes extension of our positive result to the general case.
Indeed, such an extension would readily imply
that \WL3 distinguishes any vertex-transitive graph from any non-vertex-transitive
graph, contradicting our lower bound $k=\Omega(\sqrt n)$.
This bound as well excludes any other combinatorial approach
to recognizing vertex-transitivity as long as it is based solely on~\kwl.

Our lower bound is based on the Cai-Fürer-Immerman construction \cite{CaiFI92}.
We need graphs $G$ and $H$ such that $G$ is vertex-transitive, $H$ is not,
and distinguishing $G$ and $H$ is hard for \kwl.
The CFI construction uses a special gadget to convert
an expander graph $F$ into a pair of non-isomorphic graphs $A$ and $B$
hard for \kwl, where $k$ depends on the
expansion quality of $F$. Assuming that both $A$ and $B$ are
vertex-transitive, the desired pair of $G$ and $H$ can be obtained
by taking the vertex-disjoint union of two copies
of $A$ as $G$ and the vertex-disjoint union of $A$ and $B$ as $H$.
One can expect that the vertex-transitivity of $A$ and $B$ can be
achieved whenever the template graph $F$ is vertex-transitive.
This natural idea was suggested
in the context of vertex-transitive coherent configurations 
by Evdokimov in his thesis \cite{Evdokimov04}.
An attempt to follow this scenario in the realm of uncolored graphs 
is faced with three technical complications.

First, the original CFI gadget \cite[Fig.~3]{CaiFI92} (or \cite[Fig.~13.24]{dcbook})
involves vertices of different degrees and, hence,
destroys vertex-transitivity even when the template graph $F$
is vertex-transitive. This can be overcome
by using a modified version of the CFI gadget (see Figure \ref{fig:CFI})
with all vertex degrees equal,
which apparently first appeared in \cite{EvdokimovP99}; see
also \cite{Fuerer17,NeuenS17,stacs20}. 

The second point is more subtle.
The CFI construction replaces each vertex of the template graph $F$
with a cell of new vertices, and vertices in different cells receive
different colors. In many contexts the vertex coloring can be removed
by using additional gadgets, but this is hardly possible without
losing the vertex-transitivity. The vertex colors constrain the automorphisms
of the CFI graphs $A$ and $B$ and ensure that these graphs are non-isomorphic.
We find rather general conditions on $F$ under which the CFI graphs
retain their functionality even without colors.
These conditions are satisfied by all $k$-regular graphs with $k>4$;
for $k\le 4$ we need some extra assumptions on the structure of $F$
(see Lemmas \ref{lem:twist0} and~\ref{lem:twist1}).

Finally, we need vertex-transitive template graphs $F$ with good
expansion properties. Algebraic expanders are a good choice as they are
constructed as Cayley graphs \cite{HooryLW06}. However, they allow us to show a negative
result only for rather sparse graph families. To make negative instances denser,
we use the fact shown by Babai \cite{Babai91}
that every vertex-transitive graph of sufficiently small diameter has
at least some non-trivial expansion property, which is sufficient for our purposes.

In Section \ref{s:hierarchy} we discuss a hierarchy of natural regularity properties
of graphs and show, as a consequence of failure of \kwl to recognize
vertex-transitivity, that this hierarchy does not collapse at any level.

\paragraph{Arc-transitivity.}
A graph $G$ is \emph{arc-transitive} if every ordered pair of
adjacent vertices can be taken to any other ordered pair of adjacent
vertices by an automorphism of $G$.
With just a little additional effort,
our analysis of vertex-transitivity implies that arc-transitivity of 
a given graph $G$ with a prime number of vertices can be immediately seen
from $\algt G$ and from $\algt{G_u}$ for all vertex-individualized copies of $G$
(Theorem \ref{thm:p-at}). On the negative side, 
the CFI construction can be applied with some additional care also in
the arc-transitive setting. Specifically, we prove that, if
$n$ is a square number divisible by 16, then \kwl is unable to distinguish between
arc-transitive and non-arc-transitive graphs with $n$ vertices
as long as $k=o(\sqrt n)$. The same holds true for all $n=16p$ with a prime factor $p\equiv1\pmod3$.
In this case, we need 3-regular template graphs $F$ that are
arc-transitive and have non-trivial expansion.
Requiring both arc-transitivity and 3-regularity is alone a rather strong condition.
Such graphs are extremely rare; the complete list of those with at most 512 vertices
is known as \emph{Foster census} \cite{Foster}. We employ the construction of
such graphs by Cheng and Oxley \cite{ChengO87}, and the most essential task
is here to estimate their expansion properties (cf.~Lemma~\ref{lem:diam-2}).

\section{Notation and definitions}\label{s:wl}

We denote the vertex set of a graph $G$ by $V(G)$ and the edge set by $E(G)$.
An edge $\{u,v\}\in E(G)$ can sometimes be referred to as $uv$ or $vu$.
The number of vertices in $G$ will sometimes be denoted by~$v(G)$.
The neighborhood $N(x)$ of a vertex $x$ consists of all vertices adjacent to~$x$.
We use the standard notation $K_{s,t}$ for the complete bipartite graph
where every of $s$ vertices in one part is adjacent to all $t$ vertices
in the other part. Furthermore, $2K_{s,t}$ stands for the vertex-disjoint union
of two copies of~$K_{s,t}$. The isomorphism relation is denoted by $\cong$,
and $\aut(G)$ is the automorphism group of~$G$.

\paragraph{Cayley graphs.}
Let $\Gamma$ be a group and $Z$ be a set of non-identity elements of $\Gamma$
such that $Z^{-1}=Z$, that is, any element belongs to $Z$ only together with its inverse.
The \emph{Cayley graph} $\cay(\Gamma,Z)$ has the elements of $\Gamma$ as vertices,
where $x$ and $y$ are adjacent if $x^{-1}y\in Z$. This graph is connected
if and only if the \emph{connection set} $Z$ is a generating set of $\Gamma$.
Every Cayley graph is obviously vertex-transitive.

\paragraph{The Weisfeiler-Leman algorithm.}
The original version of the Weisfeiler-Leman algorithm, \wl,
operates on the Cartesian square $V^2$ of the vertex set of an input graph $G$.
Below it is supposed that $G$ is undirected. 
We also suppose that $G$ is endowed with a vertex coloring $c$,
that is, each vertex $u\in V$
is assigned a color denoted by $c(u)$. The case of uncolored graphs is covered
by assuming that $c(u)$ is the same for all $u$.
\wl starts by assigning each pair $(u,v)\in V^2$ the initial color
$\alg20{u,v}=(\mathit{type}, c(u), c(v))$, where \textit{type} takes on one of three
values, namely \textit{edge} if $u$ and $v$ are adjacent, \textit{nonedge}
if distinct $u$ and $v$ are non-adjacent, and \emph{loop} if $u=v$.
The coloring of $V^2$ is then modified step by step. The $(r+1)$-th coloring is computed as
\begin{equation}
  \label{eq:refine-2}
  \alg2{r+1}{u,v}=\Mset{(\alg2r{u,w},\,\alg2r{w,v})}_{w\in V},
\end{equation}
where $\Mset{}$ denotes the multiset.
Let $\cS^r$ denote the partition of $V^2$ determined by the coloring $\alg2r{\cdot}$.
It is easy to notice that $\alg2{r+1}{u,v}=\alg2{r+1}{u',v'}$ implies
$\alg2r{u,v}=\alg2r{u',v'}$, which means that $\cS^{r+1}$ is finer than or equal to
$\cS^r$. It follows that the partition stabilizes starting from some step
$t\le n^2$, where $n=|V|$, that is, $\cS^{t+1}=\cS^t$, which implies
that $\cS^r=\cS^t$ for all $r\ge t$. As the stabilization is reached,
\wl terminates and outputs the coloring $\alg2t{\cdot}$, which will be denoted by~$\algt{\cdot}$.

Note that the length of $\alg2r{u,v}$ grows exponentially as $r$ increases.
The exponential blow-up is remedied by renaming the colors after each step. 

 Let $\phi$ be an automorphism of $G$.
 A simple induction on $r$ shows that
 $$\alg2r{\phi(u),\phi(v)}=\alg2r{u,v}$$
for all $r$ and, hence
\begin{equation}\label{eq:wl-inv}
  \algt{\phi(u),\phi(v)}=\algt{u,v}.
\end{equation}  
In particular, if $G$ is vertex-transitive, then the color $\algt{u,u}$ is the same
for all $u\in V$. If the last condition is fulfilled, we say that \wl
\emph{does not split the diagonal on $G$}, where by the diagonal we mean
the set of all loops~$(u,u)$.

In general, the automorphism group $\aut(G)$ of the graph $G$
acts on the Cartesian square $V(G)^2$, and the orbits of this
action are called \emph{2-orbits} of $\aut(G)$. By Equality \refeq{wl-inv},
the partition of $V(G)^2$ into 2-orbits is finer than or equal to
the stable partition $\cS=\cS^t$ produced by~\wl.

\section{Vertex-transitivity on a prime number of vertices}

We begin with a few simple observations about the output produced by \wl
on an input graph $G$.
Recall that in this paper we restrict our attention to undirected graphs.
Even though $G$ is undirected, the equality $\algt{u,v}=\algt{v,u}$
need not be true in general. Thus, the output of \wl on $G$
can naturally be seen as a complete colored directed graph on the vertex set $V(G)$,
which we denote by $\algt G$. Thus, $\algt G$ contains every pair $(u,v)\in V(G)^2$
as an arc, i.e., a directed edge, and this arc has color $\algt{u,v}$.
We will see $\algt G$ as containing no loops, but instead we assign 
each vertex $u$ the color $\algt{u,u}$.
Any directed subgraph of $\algt G$ formed by all arcs of the same color
is called a \emph{constituent digraph}.

Let $(u,v)$ and $(u',v')$ be arcs of a constituent digraph $C$ of $\algt G$.
Note that the vertices $u$ and $u'$ must be equally colored in $\algt G$.
Indeed, since the color partition of $\algt G$ is stable,
there must exist $w$ such that $(\algt{u',w},\algt{w,v'})=(\algt{u,u},\algt{u,v})$.
The equality $\algt{u',w}=\algt{u,u}$ can be fulfilled only by $w=u'$
because any non-loop $(u',w)$ is initially colored differently from the loop
$(u,u)$ and, hence, they are colored differently after all refinements.

Note also that, if $u$ and $v$ are equally colored in $\algt G$,
then they have the same outdegree in every constituent digraph $C$;
in particular, they simultaneously belong or do not belong to $V(C)$.
Otherwise, contrary to the assumption that the color partition of $\algt G$
is stable, the loops $(u,u)$ and $(v,v)$ would receive different colors in
another refinement round of \wl. It follows that 
for each constituent digraph $C$ there is an integer $d\ge1$ such that
all vertices in $C$ either have outdegree $d$ or $0$. We call $d$ the
\emph{outdegree of~$C$}.

Let $u\in V(G)$. A \emph{vertex-individualized graph} $G_u$ is obtained
from $G$ by assigning the vertex $u$ a special color, which does not occur in $G$.
If $G$ is vertex-transitive, then all vertex-individualized copies of $G$
are obviously isomorphic.

\begin{figure}
\centering
 \begin{tikzpicture}[every node/.style={circle,draw=black,inner sep=2.5pt,fill=white}]
    \FPdiv{\w}{360}{7}
    \FPdiv{\ww}{360}{14}

   \begin{scope}
  \foreach \i in {0,...,6} 
    {
    \node (x\i) at (\i*\w:\ra) {};
    }
      \foreach \i/\j in {0/2,2/4,4/6,6/1,1/3,3/5,5/0} 
        {
         \draw (x\i) -- (x\j); 
        }
      \foreach \i/\j in {0/3,3/6,6/2,2/5,5/1,1/4,4/0} 
        {
         \draw (x\i) -- (x\j); 
        }
\node[draw=none,below] () at (0mm,-15mm) {$G=\compl{C_7}$};
   \end{scope}

   \begin{scope}[xshift=40mm]
  \foreach \i in {0,...,6} 
    {
    \node[fill=black] (x\i) at (\i*\w:\ra) {};
    }
      \foreach \i/\j in {0/1,1/2,2/3,3/4,4/5,5/6,6/0} 
        {
         \draw[blue,shorten >= 2pt,shorten <= 2pt,<->] (x\i) -- (x\j); 
        }
      \foreach \i/\j in {0/2,2/4,4/6,6/1,1/3,3/5,5/0} 
        {
         \draw[red,shorten >= 2pt,shorten <= 2pt,<->] (x\i) -- (x\j); 
        }
      \foreach \i/\j in {0/3,3/6,6/2,2/5,5/1,1/4,4/0} 
        {
         \draw[green,shorten >= 2pt,shorten <= 2pt,<->] (x\i) -- (x\j); 
        }
\node[draw=none,below] () at (0mm,-15mm) {$\algt G$};
   \end{scope}

   \begin{scope}[xshift=80mm]
    \node[fill=black] (x0) at (0:\ra) {};
  \foreach \i in {1,...,6} 
    {
    \node (x\i) at (\i*\w:\ra) {};
    }
      \foreach \i/\j in {0/2,2/4,4/6,6/1,1/3,3/5,5/0} 
        {
         \draw (x\i) -- (x\j); 
        }
      \foreach \i/\j in {0/3,3/6,6/2,2/5,5/1,1/4,4/0} 
        {
         \draw (x\i) -- (x\j); 
        }
\node[draw=none,below] () at (0mm,-20mm) {$G_0$};
   \end{scope}

   \begin{scope}[xshift=120mm]
    \node[fill=black] (x0) at (0:\ra) {};
  \foreach \i in {1,6} 
    {
    \node[blue] (x\i) at (\i*\w:\ra) {};
    }
  \foreach \i in {2,5} 
    {
    \node[red] (x\i) at (\i*\w:\ra) {};
    }
  \foreach \i in {3,4} 
    {
    \node[green] (x\i) at (\i*\w:\ra) {};
    }
\FPset{\rab}{1.6}
 \draw[blue,dashed,shorten <= 3pt,shorten >= 3pt,<-] (x0).. controls (\ww:\rab) .. (x1);
 \draw[blue,shorten >= 2pt,shorten <= 2pt,->] (x0) -- (x1); 
 \draw[blue,dashed,shorten <= 3pt,shorten >= 3pt,<-] (x0).. controls (-\ww:\rab) .. (x6);
 \draw[blue,shorten >= 2pt,shorten <= 2pt,->] (x0) -- (x6);
\FPset{\rar}{1.15}
 \draw[red,dashed,shorten <= 4pt,shorten >= 4pt,<-] (x0).. controls (\w:\rar) .. (x2);
 \draw[red,shorten >= 3pt,shorten <= 3pt,->] (x0) -- (x2); 
 \draw[red,dashed,shorten <= 4pt,shorten >= 4pt,<-] (x0).. controls (-\w:\rar) .. (x5);
 \draw[red,shorten >= 3pt,shorten <= 3pt,->] (x0) -- (x5);
\FPset{\rag}{0.6}
 \draw[green,dashed,shorten <= 4pt,shorten >= 4pt,<-] (x0).. controls (3*\ww:\rag) .. (x3);
 \draw[green,shorten >= 3pt,shorten <= 3pt,->] (x0) -- (x3); 
 \draw[green,dashed,shorten <= 4pt,shorten >= 4pt,<-] (x0).. controls (-3*\ww:\rag) .. (x4);
 \draw[green,shorten >= 3pt,shorten <= 3pt,->] (x0) -- (x4);
 \draw[blue,dotted,shorten >= 2pt,shorten <= 2pt,<->] (x1) -- (x6); 
 \draw[red,dotted,shorten >= 2pt,shorten <= 2pt,<->] (x2) -- (x5); 
 \draw[green,dotted,shorten >= 2pt,shorten <= 2pt,<->] (x3) -- (x4); 
\FPset{\rarg}{1.6}
 \draw[yellow,dashed,shorten <= 3pt,shorten >= 3pt,<-] (x2).. controls (5*\ww:\rarg) .. (x3);
 \draw[yellow,shorten >= 2pt,shorten <= 2pt,->] (x2) -- (x3); 
 \draw[yellow,dashed,shorten <= 3pt,shorten >= 3pt,<-] (x5).. controls (-5*\ww:\rarg) .. (x4);
 \draw[yellow,shorten >= 2pt,shorten <= 2pt,->] (x5) -- (x4);
\FPset{\ragr}{1.1}
 \draw[magenta,dashed,shorten <= 3pt,shorten >= 3pt,<-] (x2).. controls (3*\w:\ragr) .. (x4);
 \draw[magenta,shorten >= 2pt,shorten <= 2pt,->] (x2) -- (x4); 
 \draw[magenta,dashed,shorten <= 3pt,shorten >= 3pt,<-] (x5).. controls (-3*\w:\ragr) .. (x3);
 \draw[magenta,shorten >= 2pt,shorten <= 2pt,->] (x5) -- (x3);
\FPset{\rabr}{1.6}
 \draw[brown,dashed,shorten <= 3pt,shorten >= 3pt,<-] (x1).. controls (3*\ww:\rabr) .. (x2);
 \draw[brown,shorten >= 2pt,shorten <= 2pt,->] (x1) -- (x2); 
 \draw[brown,dashed,shorten <= 3pt,shorten >= 3pt,<-] (x6).. controls (-3*\ww:\rabr) .. (x5);
 \draw[brown,shorten >= 2pt,shorten <= 2pt,->] (x6) -- (x5);
\FPset{\rarb}{0.6}
 \draw[violet,dashed,shorten <= 3pt,shorten >= 3pt,<-] (x1).. controls (-\ww:\rarb) .. (x5);
 \draw[violet,shorten >= 2pt,shorten <= 2pt,->] (x1) -- (x5); 
 \draw[violet,dashed,shorten <= 3pt,shorten >= 3pt,<-] (x6).. controls (\ww:\rarb) .. (x2);
 \draw[violet,shorten >= 2pt,shorten <= 2pt,->] (x6) -- (x2);
\FPset{\rabg}{1.1}
 \draw[black,dashed,shorten <= 3pt,shorten >= 3pt,<-] (x1).. controls (2*\w:\rabg) .. (x3);
 \draw[black,shorten >= 2pt,shorten <= 2pt,->] (x1) -- (x3); 
 \draw[black,dashed,shorten <= 3pt,shorten >= 3pt,<-] (x6).. controls (-2*\w:\rabg) .. (x4);
 \draw[black,shorten >= 2pt,shorten <= 2pt,->] (x6) -- (x4);
\FPset{\ragb}{0.6}
 \draw[cyan,dashed,shorten <= 3pt,shorten >= 3pt,<-] (x1).. controls (5*\ww:\ragb) .. (x4);
 \draw[cyan,shorten >= 2pt,shorten <= 2pt,->] (x1) -- (x4); 
 \draw[cyan,dashed,shorten <= 3pt,shorten >= 3pt,<-] (x6).. controls (-5*\ww:\ragb) .. (x3);
 \draw[cyan,shorten >= 2pt,shorten <= 2pt,->] (x6) -- (x3);

\node[draw=none,below] () at (0mm,-15mm) {$\algt{G_0}$};
   \end{scope}
\end{tikzpicture}
\caption{The output of \wl on input $\compl{C_7}$ and on its vertex-individualized copy.}
\label{fig:C7}
\end{figure}
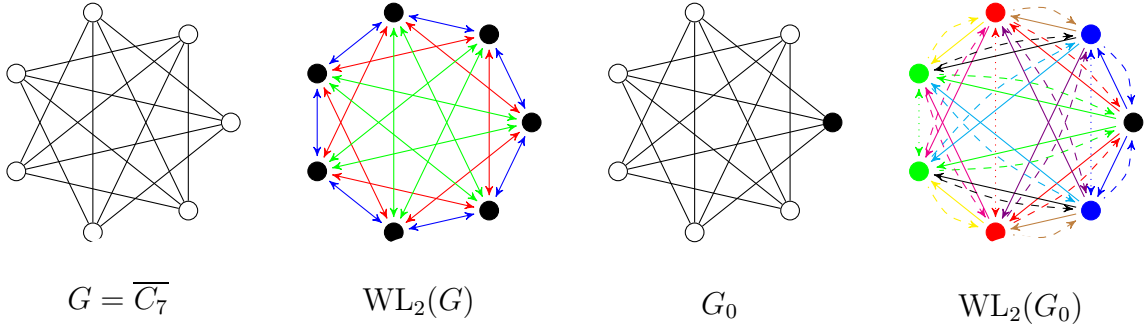

Consider now a simple and still instructive example.
Let $G=\compl{C_7}$ be the complement of the cycle graph
on seven vertices $0,1,\ldots,6$ passed in this order. It is not hard to see that \wl
splits $V(G)^2$ into the four 2-orbits of $\aut(G)$; the diagonal is one of them. 
Note that the three constituent digraphs of $\algt G$ are of the same degree 2;
see Figure \ref{fig:C7}.
Applying \wl to the vertex-individualized graph $G_0$, it can easily
be seen that \wl again splits $V(G)^2$ into the 2-orbits of $\aut(G_0)$.
Note that $\algt{G_0}$ also has exactly three constituent digraphs of outdegree 2,
while all other constituent digraphs of $\algt{G_0}$ have outdegree 1.
We see that the outdegrees of the constituent digraphs for $G$ and its
vertex-individualized copies are distributed similarly.
This similarity proves to be a characterizing property
of vertex-transitive graphs on a prime number of vertices.

\begin{theorem}\label{thm:p}
  Let $p$ be a prime, and $G$ be a graph with $p$ vertices.
  Suppose that $G$ is neither complete nor empty. Then $G$
  is vertex-transitive if and only if the following conditions are true:
  \begin{enumerate}
  \item\label{item:1}
    If run on $G$, \wl does not split the diagonal, that is,
    all vertices in $\algt G$ are equally colored.
  \item\label{item:2}
    All constituent digraphs of $\algt G$ have the same outdegree $d$ and,
    hence, there are $\frac{p-1}d$ constituent digraphs.
  \item\label{item:3}
    For every $u\in V(G)$, exactly $\frac{p-1}d$ constituent digraphs
    in $\algt{G_u}$ have outdegree $d$, and all others have outdegree~$1$.    
  \end{enumerate}  
\end{theorem}

Since the color partition of $\algt G$ for a $p$-vertex graph $G$
can be computed in time $O(p^3\log p)$ \cite{ImmermanL90},
Conditions \ref{item:1}--\ref{item:3}
can be verified in time $O(p^4\log p)$, which yields an algorithm of
this time complexity for recognition of vertex-transitivity of graphs
with a prime number of vertices.

As it will be discussed in Remark \ref{rem:29} below, there are graphs $G$ and $H$ with
a prime number of vertices such that $G$ is vertex-transitive, $H$ is not,
and still they are indistinguishable by \wl. This implies
that Theorem \ref{thm:p} is optimal in that it uses as small WL dimension
as possible and, also, that the condition involving the vertex individualization
cannot be dropped. Note that \WL1, which stands for the classical degree
refinement, does not suffice even when run on $G$ and all $G_u$
because the output of \WL1 on these inputs is subsumed by the output
of \wl on $G$ alone.

Theorem \ref{thm:p} is proved in Subsection \ref{ss:proof-p}.
The next subsection provides the necessary preliminaries.

\subsection{Coherent configurations}\label{ss:ccs}

A detailed treatment of the material presented below
can be found in \cite{CP2019}.
The stable partition $\cS=\cS^t$ of $V(G)^2$ produced by \wl
on an input graph $G$ has certain regularity properties,
which are equivalent to saying that the pair $(V,\cS)$ forms
a coherent configuration. This concept, introduced by Higman \cite{Higman71},
is defined as follows.

A \emph{coherent configuration} $\cX=(V,\cS)$ is formed
by a set $V$, whose elements are called \emph{points},
and a partition $\cS=\{S_1,\ldots,S_m\}$ of the Cartesian square $V^2$,
that is, $\bigcup_{i=1}^mS_i=V^2$ and any two $S_i$ and $S_j$ are disjoint.
An element $S_i$ of $\cS$ is referred to as a \emph{basis relation} of $\cX$.
The partition $\cS$ has to satisfy the following three conditions:
\begin{enumerate}[(A)]
\item\label{item:A}
If a basis relation $S\in\cS$ contains a loop $(u,u)$, then
all pairs in $S$ are loops.
\item \label{item:B}
For every $S\in\cS$, the \emph{transpose relation} $S^*=\Set{(v,u)}{(u,v)\in S}$
is also in~$\cS$.
\item\label{item:C}
For each triple $R,S,T\in\cS$, the number $p(u,v)=|\Set{w}{(u,w)\in R,\,(w,v)\in S}|$
for a pair $(u,v)\in T$ does not depend on the choice of this pair in~$T$. 
\end{enumerate}
In other words, if $\cS$ is seen as a color partition of $V^2$,
then such a coloring is stable under \wl refinement.

We describe two important sources of coherent configurations.
Let $\cT$ be an arbitrary family of subsets of the Cartesian square $V^2$.
There exists a unique coarsest partition $\cS$ of $V^2$ 
such that every $T\in\cT$ is a union of elements of $\cS$ and $\cX=(V,\cS)$
is a coherent configuration; see \cite[Section 2.6.1]{CP2019}. We call $\cX=(V,\cS)$ 
the \emph{coherent closure} of $\cT$ and denote it by $\clo\cT$. 

Given a vertex-colored undirected graph $G$ on the vertex set $V$,
let $\cT$ consist of the set of the pairs $(u,v)\in V^2$ such that
$\{u,v\}$ is an edge of $G$ and the sets of loops $(u,u)$ for
all vertices $u$ of the same color in $G$. Then $\clo\cT$
is exactly the stable partition produced by \wl on input $G$.
We denote this coherent configuration by~$\clo G$.

Given a coherent configuration $\cX=(V,\cS)$ and a point $u\in V$,
the coherent configuration $\cX_u=\clo{\cS\cup\{\{(u,u)\}\}}$
is called a \emph{one-point extension} of $\cX$.
This concept is naturally related to the notion of a vertex-individualized
graph, in that
\begin{equation}\label{eq:ind-ext}
\clo{G_u}=\clo{G}_u.
\end{equation}

Another source of coherent configurations is as follows.
Let $K$ be a permutation group on a set $V$.
Denote the set of 2-orbits of $K$ by $\cS$.
Then $\cX=(V,\cS)$ is a coherent configuration,
which we denote by $\inv(K)$.
Coherent configurations obtained in this way are
said to be \emph{schurian}.

We define an \emph{automorphism} of a coherent configuration $\cX=(V,\cS)$
as a bijection $\alpha$ from $V$ onto itself such that, for every $S\in\cS$
and every $(u,v)\in S$, the pair $(\alpha(u),\alpha(v))$ also belongs to $S$.
The group of all automorphisms of $\cX$ is denoted by $\aut(\cX)$.
A coherent configuration $\cX$ is schurian if and only if
\begin{equation}\label{eq:schurity}
\cX=\inv(\aut(\cX)).
\end{equation}
Note also that the connection between the coherent closure of a graph
and \wl implies that
\begin{equation}\label{eq:aut-clo}
\aut(\clo G)=\aut(G).
\end{equation}

A set of points $X\subseteq V$ is called a \emph{fiber} of $\cX$
if the set of loops $\Set{(x,x)}{x\in X}$ is a basis relation of~$\cX$.
Denote the set of all fibers of $\cX$ by $F(\cX)$.
By Property \ref{item:A},
$F(\cX)$ is a partition of $V$. Property \ref{item:C} implies that
for every basis relation $S$ of $\cX$ there are, not necessarily distinct, 
fibers $X$ and $Y$ such that $S\subseteq X\times Y$.
We use the notation $\outn xS=\Set{y}{(x,y)\in S}$ for the set of
all points in $Y$ that are in relation $S$ with $x$.
Note that $|\outn xS|=|\outn{x'}S|$ for any $x,x'\in X$.
We call this number the \emph{valency} of $S$.
If every basis relation $S$ of $\cX$ has valency 1, then $\cX$
is called \emph{semiregular}.

\begin{proposition}[{see \cite[Exercise 2.7.35]{CP2019}}]\label{prop:semireg}
A semiregular coherent configuration is schurian.
\end{proposition}

Given a set of points $U\subseteq V$ that is a union of fibers,
let $\cS_U$ denote the set of all basis relations $S\in\cS$
such that $S\subseteq X\times Y$ for some, not necessarily distinct,
fibers $X\subseteq U$ and $Y\subseteq U$. As easily seen,
$\cX_U=(U,\cS_U)$ is a coherent configuration.

If a coherent configuration has a single fiber, it is called \emph{association scheme}.

\subsection{Proof of Theorem \ref{thm:p}}\label{ss:proof-p}

\paragraph{Necessity.}
Given a vertex-transitive graph $G$ with $p$ vertices, where $p$ is prime,
we have to check Conditions \ref{item:1}--\ref{item:3}.
Condition \ref{item:1} follows immediately from
vertex-transitivity; see the discussion in the end of Section \ref{s:wl}.
For Condition \ref{item:2}, we use two basic results on vertex-transitive graphs with
a prime number of vertices. First, every such graph is isomorphic to a \emph{circulant
graph}, i.e.,
a Cayley graph of a cyclic group, because every transitive group of permutations
of a set of prime cardinality $p$ contains a $p$-cycle (Turner \cite{Turner67}). 
Let $\fp$ denote the $p$-element field, $\fpa$ its additive group, i.e.,
the cyclic group of order $p$, and $\fpm$ its multiplicative group, which is
isomorphic to the cyclic group of order $p-1$.
Another useful fact (Alspach \cite{Alspach73}) is that,
if a set $Z\subset\fp$ is non-empty and $Z\ne\fpm$, then the automorphism group of the circulant
graph $\cay(\fpa,Z)$ consists of the permutations
\begin{equation}\label{eq:autom}
x\mapsto ax+b,\,x\in\fp,\text{ for all }a\in M,\,b\in\fpa,
\end{equation}
where $M=M(Z)$ is the largest subgroup of $\fpm$ of even order such that $Z$
is a union of cosets of $M$. This subgroup is well defined because
the condition $Z=-Z$ implies that $Z$ is split into pairs $\{z,-z\}$
and, hence, is a union of cosets of the multiplicative subgroup $\{1,-1\}$.
For example, $\compl{C_7}=\cay(\bF_7^+,\{2,3,4,5\})$ and $M(\{2,3,4,5\})=\{1,-1\}$.
Without loss of generality we assume that $G=\cay(\fpa,Z)$ and denote $K=\aut(G)$.

Let $\cX=(\fpa,\cS)$ be the coherent closure of $G$.
Recall that $\cS$ is exactly the stable partition of $V(G)^2$ produced
by \wl on input $G$. The irreflexive basis relations of $\cX$ are exactly the constituent
digraphs of $\algt{G}$, and we have to prove that all of them have the same valency.

Condition \ref{item:1} says that $\cX$ is an association scheme.
In general, not all association schemes with a prime number of points are
schurian (see, e.g., \cite[Section 4.5]{CP2019}). Nevertheless,
the theorem by Leung and Man on the structure of Schur rings over cyclic groups
implies the following fact.

\begin{proposition}[{see \cite[Theorem 4.5.1]{CP2019}}]\label{prop:trans-schur}
  Let $\cX=(V,\cS)$ be an association scheme with a prime number of points.
  If $\aut(\cX)$ acts transitively on $V$, then $\cX$ is schurian.
\end{proposition}

By Equality \refeq{aut-clo}, $\aut(\cX)=K$. Since the group $K$ is transitive,
Proposition \ref{prop:trans-schur} implies that $\cX$ is schurian, and we have
$\cX=\inv(K)$ by Equality \refeq{schurity}. This yields Condition \ref{item:2} for $d=|M|$.
Indeed, every irreflexive basis relation $S\in\cS$ has valency $|M|$.
To see this, it is enough to count the number of pairs $(0,y)$ in $S$.
Fix an arbitrary pair $(0,y)\in S$. A pair $(0,y')$ is in the 2-orbit
containing $(0,y)$ if and only if $y'=ay$ for $a\in M$, for which we have $|M|$ possibilities.

It remains to prove Condition \ref{item:3}.
By vertex-transitivity, all vertex-individualized
copies of $G$ are isomorphic and, therefore, it is enough to consider $G_0$.
We have to count the frequencies of valencies in $\cX_0$. Note that $\cX_0=\clo{G_0}$
by Equality \refeq{ind-ext}.

It is generally not true that a one-point extension of a schurian coherent
configuration is schurian; see \cite[Section 3.3.1]{CP2019}. Luckily,
this is the case in our setting. The proof of the following statement
is rather involved.

\begin{proposition}[{see \cite[Theorem 4.4.14]{CP2019}}]\label{prop:ext-cyclotomic}
  If $\cX=\inv(K)$, where $K$ is the group of permutations of the form \refeq{autom}
  for a subgroup $M$ of $\fpm$, then the one-point extension $\cX_0$ is schurian.
\end{proposition}

Taking into account Equality \refeq{aut-clo}, we have
$$
\aut(\cX_0)=\aut(\clo{G_0})=\aut(G_0)=\aut(G)_0=K_0,
$$
where $K_0$ is the one-point stabilizer of $0$ in $K$, that is,
the subgroup of $K$ consisting of all
permutations $\alpha\in K$ such that $\alpha(0)=0$. Obviously,
$K_0=\Set{x\mapsto ax,\,x\in\fp}_{a\in M}$.

Let $S$ be a 2-orbit of $K_0$. If $S$ contains a pair $(0,y)$, then it
consists of all pairs $(0,y')$ for $y'\in My$ and, hence, has valency $|M|$.
If $S$ contains a pair $(z,y)$ with $z\ne0$ and $y\ne z$, then $(z,y)$ is the only
element of $S$ with the first coordinate $z$, and $S$ has valency 1.
The proof of Condition \ref{item:3} is complete.

\paragraph{Sufficiency.}
Let $G$ be a graph satisfying Conditions \ref{item:1}--\ref{item:3} stated in the theorem.
Let $\cX=\clo G$. Condition \ref{item:1} says that $\cX$ is an association scheme.
By Equality \refeq{aut-clo}, it suffices
to prove that the group $\aut(\cX)$ is transitive.
The proof is based on the following lemma.
Though a stronger fact is established in \cite[Theorem 7.1]{MP2012},
we here include a proof of the lemma for the reader's convenience.
We closely follow \cite{MP2012} but provide a bit more details
for non-experts in algebraic graph theory.

\begin{lemma}\label{lem:MP}
  Let $\cX=(V,\cS)$ be an association scheme. Suppose that the following
  two conditions are true for every point $u\in V$:
  \begin{enumerate}[(I)]
  \item\label{item:i}
    the coherent configuration $(\cX_u)_ {V\setminus\{u\}}$ is semiregular, and
  \item\label{item:ii}
    $F(\cX_u)=\Set{\outn uS}{S\in\cS}$.
  \end{enumerate}
  Then the group $\aut(\cX)$ acts transitively on~$V$.    
\end{lemma}

\begin{proof}
  Let $u\in V$. By Condition \ref{item:i} and Proposition \ref{prop:semireg},
  the coherent configuration $\cX'=(\cX_u)_ {V\setminus\{u\}}$ is schurian.
  We claim that the coherent configuration $\cX_u=(V,\cS_u)$
  is schurian also. Let $K'=\aut(\cX')$ and define $K$ to be a permutation
  group on $V$ that, for each $\alpha'\in K'$, contains the permutation $\alpha$
  which fixes $u$ and coincides with $\alpha'$ on $V\setminus\{u\}$.
  Consider a basis relation $S\in\cS_u$. If $S$ is a basis relation of $\cX'$,
  then it is a 2-orbit of $K'$ and, hence, of $K$. If $S$ is not a basis relation of $\cX'$,
  then $S=\{u\}\times X$ or $S=X\times\{u\}$ for a fiber $X\in F(\cX')$, or $S=\{(u,u)\}$. 
Since $X$ is an orbit of $K'$ and,
  hence, of $K$, the set $S$ is a 2-orbit of $K$ also in this case. Thus, $\cX_u=\inv(K)$.

  Let $A=\aut(\cX)$. Since $\aut(\cX_u)=\aut(\cX)_u=A_u$
  (where $A_u$ denotes the one-point stabilizer of $u$ in $A$)
  and $\cX_u$ is schurian, $\cX_u=\inv(A_u)$. Condition \ref{item:ii}
  implies that, for every $S\in\cS$, $\outn uS$ is an orbit of~$A_u$. 

  Denote the orbit of $A_u$ containing a point $x$ by $A_u(x)$.
  We claim that $|A_u(x)|=|A_u|$ for every point $x\ne u$.
  Indeed, if $\alpha(x)=\beta(x)$ for two different permutations in $A_u$,
  then the stabilizer $(A_u)_x$ contains a non-identity permutation,
  namely $\alpha^{-1}\beta$. Therefore, the 2-orbit of $A_u$ containing
  a pair $(x,y)$ with $y\notin\{u,x\}$ contains also the pair $(x,\alpha^{-1}\beta(y))$.
  Choosing $y$ to be a point moved by $\alpha^{-1}\beta$, we get a contradiction
  with Condition~\ref{item:i}.

  We conclude that $|\outn uS|=|A_u|$ for every irreflexive $S\in\cS$.
  The right hand side of this equality
  does not depend on $S$. The left hand side does not depend on $u$
  because $\cX$ is an association scheme. It follows that there exists
  a number $k$ such that $|A_u|=k$ for all $u\in V$.

  Let $Q$ be an orbit of $A$. Since $A_u$ is a subgroup of $A$,
  $Q$ is a union of orbits of $A_u$. Choosing a point $u$ in $Q$,
  we conclude that $|Q|\equiv1\pmod k$. If there was a point $u'\notin Q$,
  the same argument would yield $|Q|\equiv0\pmod k$. Therefore,
  such $u'$ does not exist, and $Q=V$, as desired.
\end{proof}

Let $u$ be an arbitrary vertex of $G$. By Equality \refeq{ind-ext}, $\cX_u=\clo{G_u}$.
Now, it suffices to derive Conditions \ref{item:i}--\ref{item:ii} in the lemma
from Conditions \ref{item:1}--\ref{item:3} in the theorem.

For a fiber $X\in F(\cX_u)$, note that $\{u\}\times X$ must be a basis relation
of $\cX_u$. Since this relation has valency $|X|$, Condition \ref{item:3}
implies that every fiber in $F(\cX_u)$ is either a singleton or consists
of $d\ge2$ points. Denote the number of singletons in $F(\cX_u)$ by $a$.
Besides of them, $F(\cX_u)$ contains $(p-a)/d$ fibers of size $d$.

For every $X,Y\in F(\cX_u)$ with $|X|=1$ and $|Y|=d$,
$X\times Y$ is a basis relation of $\cX_u$ of valency $d$.
It follows from Condition \ref{item:3} that
$$
\frac{p-1}{d}\ge\frac{a(p-a)}{d}.
$$
Therefore, $p-1\ge a(p-a)$ or, equivalently, $p(a-1)\le(a-1)(a+1)$.
Assume for a while that $a>1$. It immediately follows that $a\ge p-1$.
Since the equality $a=p-1$ is impossible, we conclude that $a=p$.
However, this implies that $d=1$, a contradiction. Thus, $a=1$.
Consequently, every fiber of the coherent configuration $\cX'=(\cX_u)_ {V\setminus\{u\}}$
is of cardinality $d$, and $|F(\cX')|=(p-1)/d$.

Let $S$ be a basis relation of $\cX$. If $S$ is reflexive, then $\outn uS=\{u\}$.
If $S$ is irreflexive, then $\outn uS$ must be a union
of fibers in $F(\cX')$. By Condition \ref{item:2},
the number of irreflexive basis relations in $\cS$ is $(p-1)/d$. 
It follows that $\outn uS$ actually coincides with one of the fibers of $\cX'$.
This proves Condition~\ref{item:ii}.

Since $\cX_u$ contains $(p-1)/d$ basis relations of the kind $\{u\}\times X$
for $X\in F(\cX')$, Condition \ref{item:3} implies that every basis relation
of $\cX'$ is of valency $1$, yielding Condition~\ref{item:i}. 

The proof of Theorem \ref{thm:p} is complete.

\begin{remark}\label{rem:29}
We now argue that there is a vertex transitive graph $G$ and a non-vertex-transitive
graph $H$ such that $G$ and $H$ are indistinguishable by \wl.
Recall that a \emph{strongly regular graph} with parameters $(n,d,\lambda,\mu)$  
is an $n$-vertex $d$-regular graph where every two adjacent vertices
have $\lambda$ common neighbors, and every two non-adjacent vertices
have $\mu$ common neighbors. As easily seen, two strongly regular graphs with
the same parameters are indistinguishable by \wl, and our example
will be given by $G$ and $H$ of this kind.
Let $p$ be a prime (or a prime power) such that $p\equiv1\pmod4$.
The \emph{Paley graph} on $p$ vertices is the Cayley graph
$\cay(\fpa,Y_p)$ where $Y_p$ is the subgroup of $\fpm$
formed by all quadratic residues modulo $p$.
The assumption $p\equiv1\pmod4$ ensures that $-1$ is a quadratic residue modulo $p$
and, hence, $Y_p=-Y_p$. The Paley graph on $p$ vertices is
strongly regular with parameters $(p,\frac{p-1}2,\frac{p-5}4,\frac{p-1}4)$.

Let $G$ be the Paley graph on $29$ vertices. It is known (Bussemaker and Spence;
see, e.g., \cite[Section 9.9]{BrouwerH12}) that there are 40 other strongly regular graphs
with parameters $(29,14,6,7)$. Let $H$ be one of them.
We have only to show that $H$ is not vertex-transitive.
Otherwise, by \cite{Turner67} this would be a circulant graph, that is,
we would have $H=\cay(\fpa,Z)$ for some connection set $Z$.
In this case, the coherent closure $\clo H$ must be schurian by Proposition \ref{prop:trans-schur}.
Since $H$ is strongly regular, \wl colors all pairs of adjacent vertices
uniformly and, therefore, they form a 2-orbit of $\aut(H)$.
It follows that the stabilizer $\aut(H)_0$ acts transitively
on $N(0)$, the neighborhood of 0 in $H$. The aforementioned result of Alspach, 
implies that $Z$ is the subgroup of $\fpm$ of order $(p-1)/2$
i.e., $M=Z$ in \refeq{autom}. This means that $Z=Y_p$ and $H=G$, a contradiction.
\end{remark}

\section{A lower bound for the WL dimension}\label{s:vt}

We now prove a negative result on the recognizability of
vertex-transitivity by \kwl.
We begin with a formal definition of the $k$-dimensional algorithm.
Let $k\ge 2$. Given a graph $G$ with vertex set $V$ as input,
\kwl operates on $V^k$. The initial
coloring of $\baru=(u_1,\ldots,u_k)$ encodes the equality type of this $k$-tuple
and the ordered isomorphism type of the subgraph of $G$ induced by
the vertices $u_1,\ldots,u_k$. 
The color refinement is performed similarly to \refeq{refine-2}.
Specifically, \kwl iteratively colors $V^k$ by
$\alg k{r+1}\baru=\Mset{
  (\alg kr{\baru_1^w},\dots,\alg kr{\baru_k^w})
}{w\in V(G)}$,
where $\baru_i^w=(u_1,\dots,u_{i-1},w,u_{i+1},\dots,u_k)$. If $G$ has $n$ vertices,
the color partition stabilizes in $t\le n^k$ rounds, and \kwl
outputs the coloring $\kalgt{\cdot}=\alg kt{\cdot}$.

We say that \kwl \emph{distinguishes} graphs $G$ and $H$ if the
final color palettes are different for $G$ and $H$, that is,
$\Mset{\kalgt\baru}{\baru\in V(G)^k}\ne\Mset{\kalgt\baru}{\baru\in V(H)^k}$
(note that color renaming in each refinement round must be the same on $G$ and $H$).

\begin{theorem}\label{thm:vt}\hfill
  \begin{enumerate}[\bf 1.]
    \item
  For every $n$ divisible by 16 there are $n$-vertex graphs $G$ and $H$
  such that $G$ is vertex-transitive, $H$ is not, and $G$ and $H$
  are indistinguishable by \kwl as long as $k\le0.01\,\sqrt n$.
\item
For infinitely many $n$ there are $n$-vertex graphs $G$ and $H$
  such that $G$ is vertex-transitive, $H$ is not, and $G$ and $H$
  are indistinguishable by \kwl as long as $k\le0.001\,n$.
  \end{enumerate}
\end{theorem}

The proof is given in Subsection \ref{ss:proof}. 
As discussed in Section \ref{s:intro}, it is based on
a regularized and discolored version of the Cai-Fürer-Immerman construction \cite{CaiFI92},
whose description and analysis takes Subsection~\ref{ss:CFI}.

\subsection{The CFI graphs (regular and colorless)}\label{ss:CFI}

Given a $k$-regular template graph $F$ on $m$ vertices, we construct a graph $A$
with $m2^{k-1}$ vertices by replacing each vertex $v$ of $F$ with a cell $Q(v)$ consisting
of $2^{k-1}$ new vertices. Each cell is endowed with a hypergraph structure
which, being a direct analog of the standard CFI gadget, will determine
the adjacency relation of~$A$.

Specifically, let $Q\subset\{0,1\}^k$ be the set of all boolean vectors
of length $k$ with even number of ones. That is, $(x_1,\ldots,x_k)\in Q$
if and only if $x_1\oplus\ldots\oplus x_k=0$, where $\oplus$ denotes
the addition in $\bZ_2$. Note that $|Q|=2^{k-1}$.
For $i\le k$, the \emph{$i$-th slice partition} of $Q$
consists of two parts $X_i^0$ and $X_i^1$, where $X_i^b=\Set{(x_1,\ldots,x_k)\in Q}{x_i=b}$.

Let $\cX$ be the hypergraph on the vertex set $Q$ with the hyperedges $X_i^b$
for $i\le k$ and $b=0,1$. The slice partitions are interchangeable as,
for each pair $1\le i<j\le k$, there is an automorphism $\alpha_{ij}$ of $\cX$
such that $\alpha_{ij}(X_i^b)=X_j^b$, $\alpha_{ij}(X_j^b)=X_i^b$, and $\alpha_{ij}(X_l^b)=X_l^b$
for any $l\ne i,j$. Specifically, $\alpha_{ij}$ swaps the $i$-th and the $j$-th positions in
each vector $(x_1,\ldots,x_k)\in Q$.

We say that an automorphism $\lambda$ of $\cX$ \emph{preserves} the $i$-th slice partition
if either $\lambda(X_i^b)=X_i^b$ for both $b=0,1$ (the partition is \emph{fixed})
or $\lambda(X_i^b)=X_i^{1-b}$ for both $b=0,1$ (the partition is \emph{flipped}).
Let $\Lambda_Q$ denote the subgroup of $\aut(\cX)$ that consists of all $\lambda$
preserving each of the $k$ slice partitions. For $s\in Q$, define a permutation $\lambda_s$ of $Q$
by $\lambda_s(x)=x\oplus s$, where $\oplus$ denotes the addition in the group $(\bZ_2)^k$.
Note that $\Lambda_Q=\Set{\lambda_s}{s\in Q}$.\footnote{A generalized CFI construction
uses a cell $Q_p\subset\bZ_p$, for any prime $p$, consisting of those vectors $(x_1,\ldots,x_k)$
with $x_1\oplus\ldots\oplus x_k=0$, where $\oplus$ is the addition modulo $p$.
The $i$-th slice partitions of $Q_p$ is defined similarly to the case $p=2$
and consists of $p$ parts. Unlike the case $p=2$, the group of the partition-preserving
automorphisms of $Q_p$ for $p>2$ can be larger than the group of shifting automorphisms $\lambda_s$. 
An analysis of the colorless version of the generalized construction is 
out the scope of this paper.}

Let $Q(v)$ be a copy of $Q$. We relabel the coordinates $1,\ldots,k$ using
the $k$ neighbors of $v$ in $F$. The slice partitions are now
$Q(v)=X^0_{u}(v)\cup X^1_{u}(v)$ for $u\in N(v)$.
The adjacency relation in $A$ is defined as follows.
For each pair of adjacent vertices $u$ and $v$ in $F$ and for each $b=0,1$,
we connect every vertex in $X^b_{u}(v)$ by an edge to every vertex in $X^b_{v}(u)$. 
In other terms, two vertices $x\in Q(v)$ and $y\in Q(u)$ are adjacent if $x_u=y_v$.

For two disjoint sets of vertices $X$ and $X'$ in a graph $A$,
we write $A[X,X']$ to denote the bipartite graph with vertex classes $X$ and $X'$
and all edges from $E(A)$ between $X$ and~$X'$. Thus, $A[X^b_{u}(v),X^b_{v}(u)]\cong K_{2^{k-2},2^{k-2}}$
is a complete bipartite graph with both vertex classes of size $2^{k-2}$,
und $A[Q(v),Q(u)]$ is the vertex-disjoint union of two such graphs.
This completes description of the graph $A$. Note that every vertex in $A$
has degree $k2^{k-2}$.
The subgraph $A[Q(v),Q(u)]$ for any pair of adjacent $v$ and $u$ will be referred to
as an \emph{interspace} of~$A$.
Figure~\ref{fig:CFI} shows a cell of $A$ along with three incident interspaces in the case of $k=3$.

\begin{figure}
\centering
\begin{tikzpicture}[
  every node/.style={circle,draw,black,
  inner sep=2pt,fill=black},
  elabel/.style={black,draw=none,fill=none,rectangle},
  every edge/.append style={every node/.append style={elabel}},
  lab/.style={draw=none,fill=none,inner sep=0pt,rectangle},
  ri/.style={label position=right},
  le/.style={label position=left},
  line width=0.3pt,
]
  \matrixgraph[name=m1,nolabel]
    {&[2mm]&[3mm]&[2mm]&[2mm]&[9mm]&[9mm]&[9mm]
    &[2mm]&[2mm]&[3mm]&[2mm]\\
    &&&&    y_1&y_2&y_3&y_4        \\[12mm]
    &&&&    x_1&x_2&x_3&x_4        \\[3mm]
     z_4[le] &&&&&&  &&&&& u_4[ri]\\[2mm]
    & z_3[le] &&&& & &&&& u_3[ri]\\[3mm]
    && z_2[le] &&&&  &&& u_2[ri]\\[2mm]
    &&& z_1[le] && & && u_1[ri]\\
  }{
    {z_1,z_2} -- {x_1,x_4};
    {z_3,z_4} -- {x_2,x_3};
    {u_1,u_2} -- {x_1,x_2};
    {u_3,u_4} -- {x_3,x_4};
    {y_1,y_3} -- {x_1,x_3};
    {y_2,y_4} -- {x_2,x_4};
  };
  \colclass[0,fill=gray!30]{y_1,y_2,y_3,y_4}
  \colclass[0,fill=gray!30]{x_1,x_2,x_3,x_4}
  \colclass[-46,fill=gray!30]{z_1,z_2,z_3,z_4}
  \colclass[46,fill=gray!30]{u_1,u_2,u_3,u_4}
\end{tikzpicture}
\caption{The regularized CFI gadget for $k=3$.}
\label{fig:CFI}
\end{figure}
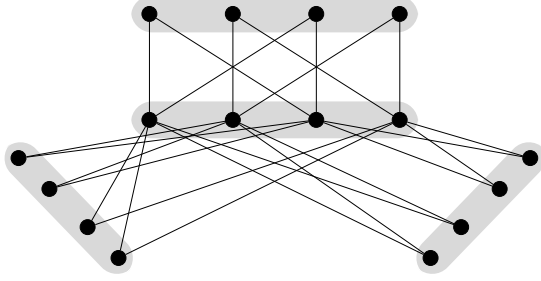

The adjacency relation of $A$ depends on the labeling of slice partitions for each $Q(v)$,
but the choice of a particular labeling is immaterial due to the interchangeability
of slice partitions in $Q$. Even when a labeling is fixed, an interspace $A[Q(v),Q(u)]$
is not determined unambiguously due to the automorphisms of $\cX$ in $\Lambda_Q$.
Specifically, a \emph{twist} of $A[Q(v),Q(u)]$ consists in reconnecting every vertex in 
$X^b_{u}(v)$ by an edge to every vertex in $X^{1-b}_{v}(u)$ for each $b=0,1$.
Note that a twisted interspace remains isomorphic to $2K_{2^{k-2},2^{k-2}}$.
If we repeat the construction with a cell $Q(v)$ permuted according to
$\lambda_s$ in $\Lambda_Q$, this will result in the modification of $A$
where the interspaces $A[Q(v),Q(u)]$ are twisted for all those $u\in N(v)$ with $s_u=1$.

Fix a graph $A$ constructed as described above.
For a set $S\subseteq E(F)$,
let $A^S$ denote the graph obtained from $A$ by twisting the subgraph $A[Q(v),Q(u)]$
for all edges $vu\in S$. 
For the original CFI construction, where each cell has its own color that must 
be preserved by any isomorphism, \cite[Lemma 6.2]{CaiFI92} says that
\begin{equation}
  \label{eq:AAS}
A\cong A^S\text{ if and only if }|S|\text{ is even.}
\end{equation}
In our colorless version,
$A$ can have automorphisms that do not map cells onto cells; see Remark \ref{rem:cond-on-F}.
This makes establishing the equivalence \refeq{AAS}
in the `only if' direction a somewhat subtle issue. We compensate the loss of rigidity
caused by removal of vertex colors by imposing a restriction on template graphs~$F$.

\begin{definition}
For two distinct vertices $u$ and $v$ of a graph $G$, let $\nu(u,v)=|N(u)\cap N(v)|$
denote the number of common neighbors of $u$ and $v$. The maximum possible value
of $\nu(u,v)$ for $G$ will be denoted by~$\nu(G)$.  
\end{definition}

\begin{lemma}\label{lem:twist0}
Suppose that $F$ is $k$-regular and $\nu(F)<2k-4$. Let $R,S\subseteq E(F)$.
If $\alpha$ is an isomorphism from $A^R$ to $A^S$, then there is $\phi\in\aut(F)$
such that $\alpha(Q(v))=Q(\phi(v))$ for every $v\in V(F)$.
\end{lemma}

\begin{proof}
We first determine the value of $\nu(A^R)$ (which is the same for $A^S$ because $R$
is arbitrary). Suppose that two vertices $x$ and $y$ are in the same cell $Q(v)$. 
These vertices can have a common neighbor only in a cell $Q(w)$ for $w\in N(v)$.
A vertex $z\in Q(w)$ is adjacent to $x$ exactly when 
\begin{equation}
  \label{eq:adj-cond}
  \begin{array}{lcl}
z_v=x_w&\text{and}&\{v,w\}\notin R,\text{ or}\\
z_v\ne x_w&\text{and}&\{v,w\}\in R 
  \end{array}
\end{equation}
 (in the latter case, the interpace $A^R[Q(v),Q(w)]$ is twisted).
Along with the similar condition for adjacency to $y$, this implies the following.
If $x_w\ne y_w$, then $x$ and $y$ have no neighbors in $Q(w)$.
If $x_w=y_w$, then the number of such neighbors is $2^{k-2}$.
It readily follows that $\nu(x,y)=(k-H(x,y))\,2^{k-2}$, where $H(x,y)$ denotes the Hamming
distance between $k$-vectors. Note that the maximum value of $\nu(x,y)$ over $x,y\in Q(v)$
is equal to $(k-2)\,2^{k-2}$.

Suppose now that $x\in Q(v)$ and $y\in Q(u)$ for $u\ne v$. Such $x$ and $y$
can have a common neighbor only in a cell $Q(w)$ for $w\in N(v)\cap N(u)$.
Let $z\in Q(w)$. The adjacency condition for $z$ and $x$ is expressed by \refeq{adj-cond}.
The adjacency condition for $z$ and $y$ is stated similarly but
with respect to the coordinate $\{w,u\}$ instead of $\{w,v\}$. 
It follows that the number of common neighbors of $x$ and $y$ in $Q(w)$ is equal to $2^{k-3}$.
Therefore, 
$$
\nu(x,y)=\nu(u,v)\,2^{k-3}\le\nu(F)\,2^{k-3}<(k-2)2^{k-2},
$$
where the last inequality is due to the assumption on $\nu(F)$.
We conclude that $\nu(A^R)=(k-2)2^{k-2}$.

Given a graph $G$, we define an irreflexive symmetrix relation on $V(G)$
by setting $x\approx y$ if $\nu(x,y)=\nu(G)$. Let $G_\approx$ denote the
graph on $V(G)$ with adjacency relation $\approx$. The estimates above imply
that $(A^R)_\approx$ is a graph whose connected components are exactly the cells
$Q(v)$, $v\in V(F)$. This follows from the observation that, 
for every two distinct $x,y\in Q(v)$, the Hamming distance $H(x,y)$ is even
and, hence, there is a sequence $x=x_1,\ldots,x_s=y$ in $Q$ with $H(x_i,x_{i+1})=2$,
which is a path in $(A^R)_\approx$.

Note that an isomorphism $\alpha$ from $A^R$ to $A^S$ is also an
isomorphism from $(A^R)_\approx$ to $(A^S)_\approx$.
Since any isomorphism maps a connected component onto a connected component,
we conclude that, for every $v$ there is $v'$ such that $\alpha(Q(v))=Q(v')$.
The correspondence $v\mapsto v'$ is an automorphism $\phi$ of $F$ by the construction of~$A$.
\end{proof}

\begin{remark}\label{rem:cond-on-F}
The assumption $\nu(F)<2k-4$ in Lemma \ref{lem:twist0} is true
for every $k$-regular $F$ with $k>4$ just because $\nu(F)\le\Delta(F)=k$.
For $k=3$ this assumption cannot be dropped.
For example, the $4\times4$-rook's graph is obtainable by the CFI construction from 
the template graph $K_4$, the complete graph on 4 vertices. However, the cell partition
is not preserved by all automorphisms of this graph because its complement is
arc-transitive.
\end{remark}

Assume that the assumptions of Lemma \ref{lem:twist0} are satisfied and 
apply it in the case $R=S=\emptyset$. The correspondence $\alpha\mapsto\phi$ is clearly
a group homomorphism from $\aut(A)$ to $\aut(F)$, which we denote by $h$.
Denote the kernel of $h$ by $\aut^*(A)$. Thus, $\aut^*(A)$ consists of the
automorphisms of $A$ mapping each cell onto itself. Note that $h$ is surjective.
Indeed, given $\phi\in\aut(F)$, we can construct $\alpha\in\aut(A)$ such that $h(\alpha)=\phi$
just by obeying the equalities 
\begin{eqnarray}
  \label{eq:Q}
&&\alpha(Q(v))=Q(\phi(v))\text{ and, moreover,}\\
  \label{eq:Qbu} 
&&\alpha(Q^b_u(v))=Q^b_{\phi(u)}(\phi(v))\text{ for all }v\in V(F),\ u\in N(v),\text{ and }b=0,1
\end{eqnarray}
(using the automorphisms $\alpha_{ij}$ of the hypergraph $Q$ discussed in the beginning
of this subsection). It follows that $\aut(A)/\aut^*(A)\cong\aut(F)$.

\begin{lemma}\label{lem:autRaut}
Let $R,S,T\subseteq E(F)$.
If $\alpha$ is an isomorphism from $A^R$ to $A^S$ and $\phi$ is an automorphism of $F$
such that $\alpha(Q(v))=Q(\phi(v))$ for every $v\in V(F)$, then
$\alpha$ is an isomorphism from $A^{R\symdiff T}$ to $A^{S\symdiff \phi(T)}$.
\end{lemma}

\begin{proof}
It suffices to prove that, for two vertices $v$ and $u$ adjacent in $F$,
$\alpha$ induces an isomorphism from the interspace $A^{R\symdiff T}[Q(v),Q(u)]$
to the interspace $A^{S\symdiff \phi(T)}[Q(\phi(v)),\allowbreak Q(\phi(u))]$.
By assumption, $\alpha$ is an isomorphism from $A^R[Q(v),Q(u)]$
to $A^S[Q(\phi(v)),\allowbreak Q(\phi(u))]$. If $vu\notin T$, we are done just because
$A^{R\symdiff T}[Q(v),Q(u)]=A^R[Q(v),Q(u)]$ and 
$A^{S\symdiff \phi(T)}[Q(\phi(v)),Q(\phi(u))]=A^S[Q(\phi(v)),Q(\phi(u))]$.
If $vu\in T$, then both the interspaces are twisted and $\alpha$ stays
an isomorphism between them.
\end{proof}

\begin{lemma}\label{lem:twist1}
Suppose that $F$ is a $k$-regular graph with $\nu(F)<2k-4$.
Let $R,S\subseteq E(F)$.
If $A^R\cong A^S$, then $|R\symdiff S|$ is even.
\end{lemma}

\begin{proof}
Let $\alpha$ be an isomorphism from $A^R$ to $A^S$.
Let $\phi$ be an automorphism of $F$ induced by $\alpha$ according to Lemma \ref{lem:twist0}.
Taking into account the discussion before Lemma \ref{lem:autRaut}, fix
$\beta\in\aut(A)$ such that $h(\beta)=\phi$.
By Lemma \ref{lem:autRaut} applied for $R=S=\emptyset$ and $T=R$, 
the automorphism $\beta$ is an isomorphism from $A^R$ to $A^{\phi(R)}$.
The composition $\gamma=\alpha\beta^{-1}$ is, therefore, an isomorphism from $A^{\phi(R)}$ to $A^S$
taking each cell onto itself, i.e., $\gamma(Q(v))=Q(v)$ for every $v\in V(F)$.
Let $V(F)=\{v_1,\ldots,v_m\}$. We can represent $\gamma$ as a product
$\gamma=\gamma_1\gamma_2\ldots\gamma_m$ where $\gamma_i$ is the identity outside $Q(v_i)$.
Thus, $\gamma_i$ can be seen as a permutation of $Q(v_i)$.
The condition that $\gamma$ is an isomorphism from $A^{\phi(R)}$ to $A^S$ readily implies that
$\gamma_i$ preserves the slice partitions of $Q(v_i)$, i.e.,
$\gamma_i\in\Lambda_{Q(v_i)}$ for every $i\le m$.

For any $t$ from 1 to $m$, let $\bar\gamma_t=\gamma_1\ldots\gamma_t$.
Using the induction on $t$, we will show that $\bar\gamma_t$ is an isomorphism from $A^{\phi(R)}$
to $A^{S_t}$ for a set $S_t\subseteq E(F)$ such that $|\phi(R)\symdiff S_t|$ is even. 

In the base case $t=1$, we have $\bar\gamma_1=\gamma_1$, and the claim follows from
the fact that $\gamma_1$, like any permutation in $\Lambda_Q$, twists
an even number of interspaces. It suffices to notice that, when $\gamma_1$
transforms $A^{\phi(R)}$ into $A^{S_1}$, then it twists exactly the interspaces $A^{\phi(R)}[Q(u),Q(v)]$
corresponding to $\{u,v\}\in\phi(R)\symdiff S_1$. Suppose now that $\bar\gamma_{t-1}$ is an isomorphism from $A^{\phi(R)}$
to $A^{S_{t-1}}$ with $|\phi(R)\symdiff S_{t-1}|$ even. Let $S'\subseteq E(F)$ correspond to the set
of the interspaces twisted by $\gamma_t$. Note that $\gamma_t$ transforms $A^{S_{t-1}}$ into
$A^{S_{t-1}\symdiff S'}$ and, hence, $\bar\gamma_t$ transforms $A^{\phi(R)}$ into $A^{S_t}$ where $S_t=S_{t-1}\symdiff S'$.
Since both $\phi(R)\symdiff S_{t-1}$ and $S'$ are of even size, their symmetric difference 
$(\phi(R)\symdiff S_{t-1})\symdiff S'=\phi(R)\symdiff S_t$ is also of even size.

For $t=m$ we have $S_t=S$ and, therefore, $|\phi(R)\symdiff S|=|\phi(R)|+|S|-2|\phi(R)\cap S|$ is even.
Since $|\phi(R)|=|R|$, we conclude that $|R\symdiff S|=|R|+|S|-2|R\cap S|$ is even too.
\end{proof}

\begin{lemma}\label{lem:twist2}
Suppose that $F$ is connected. Let $R,S\subseteq E(F)$.
If $|R\symdiff S|$ is even, then
$A^R\cong A^S$ and, moreover, there exists an isomorphism
from $A^R$ to $A^S$ mapping every cell $Q(v)$ onto itself.
\end{lemma}

\begin{proof}
This lemma is proved similarly to \cite[Lemma 6.2]{CaiFI92}.

Note that any isomorphism from $A^R$ to $A^S$ has to twist exactly $|R\symdiff S|$ interspaces.
Suppose that $|R\symdiff S|=2n$ and use the induction on $n$.
Consider the base case $n=1$.
Assume first that the two edges in $R\symdiff S$ are adjacent, say, those are $vu$ and $uw$.
Let $\lambda_{vuw}$ denote the permutation in $\lambda_s\in\Lambda_{Q(u)}$
for the shift vector $s$ with $s_v=s_w=1$ and $s_z=0$ for all other $z\in N(u)$.
We extend $\lambda_{vuw}$ to the whole vertex set of $A^R$ by identity.
This permutation flips each of the slice partitions $\Set{Q^b_v(u)}_{b=0,1}$
and $\Set{Q^b_w(u)}_{b=0,1}$ and fixes the other slice partitions of $Q(u)$.
Thus, $\lambda_{vuw}$ twists each of the interspaces
$A^R[Q(u),Q(v)]$ and $A^R[Q(u),Q(w)]$ and does not change anything else in the graph.
We conclude that $\lambda_{vuw}$ is an isomorphism from $A^R$ to~$A^S$.

If $R\symdiff S$ consists of two non-adjacent edges $vv'$ and $ww'$,
consider a path connecting these edges, say, $vu_1\ldots u_kw$.
The permutation 
$$
\lambda_{v'vu_1}\lambda_{vu_1u_2}\lambda_{u_1u_2u_3}\ldots\lambda_{u_{k-1}u_kw}\lambda_{u_kww'}
$$
is an isomorphism from $A^R$ to $A^S$ because it twists both interspaces
$A^R[Q(v'),Q(v)]$ and $A^R[Q(w),Q(w')]$ and changes nothing else
(each of the intermediate iterspaces $A^R[Q(v),Q(u_1)]$, $A^R[Q(u_1),Q(u_2)]$ etc.\
is twisted twice).

Suppose now that $n\ge2$. Choose two edges $e_1$ and $e_2$ in $R\symdiff S$
and set $T=S\symdiff\{e_1,e_2\}$. Since $R\symdiff T=(R\symdiff S)\setminus\{e_1,e_2\}$
is of size $2(n-1)$, the induction assumption gives us a cell-preserving
isomorphism $\beta$ from $A^R$ to $A^T$. Note that $T\symdiff S=\{e_1,e_2\}$
is of size 2. According to the base case, there exists a cell-preserving isomorphism $\alpha$ 
from $A^T$ to $A^S$. The composition $\alpha\beta$ is a cell-preserving isomorphism
from $A^R$ to~$A^S$.
\end{proof}

From now on we impose the following conditions on template graphs.

\begin{assumption}\label{ass:}
A template graph $F$ is $k$-regular with $k\ge3$. 
Moreover, $F$ is connected and $\nu(F)<2k-4$.
\end{assumption}

Let $B=A^{\{e\}}$ for an edge $e$ of $F$. It follows from Lemma \ref{lem:twist2}
that the isomorphism type of $B$ does not depend on the choice of $e$.
Lemma \ref{lem:twist1} implies that $A$ and $B$ are not isomorphic.

Given $X\subset V(F)$, let $F\setminus X$ denote the graph obtained
from $F$ by removal of all vertices in $X$. We call $X$
a \emph{separator} of $F$ if every connected component of the graph
$F\setminus X$ has at most $v(F)/2$ vertices.
The number of vertices in a separator is called its {\em size}.
We denote the minimum size of a separator of $F$ by $s(F)$.
The following crucial fact is established in \cite[Theorem 6.4]{CaiFI92}
and is all the more true in the colorless version of the construction.

\begin{lemma}\label{lem:sep}
The graphs $A$ and $B$ are indistinguishable by \kwl for all $k<s(F)$.
\end{lemma}

\subsection{Proof of Theorem \ref{thm:vt}}\label{ss:proof}

We begin with an elementary property of vertex-transitive graphs.

\begin{lemma}\label{lem:vt-elem}
Every 3-path $vuw$ in a vertex-transitive graph $F$ extends to a cycle.
\end{lemma}

\begin{proof}
Let $F'$ be the connected component of $F$ containing the path $vuw$.
If $F'$ has a cut vertex, then all vertices of $F'$ must be cut vertices
because $F'$ is vertex-transitive. This is impossible because $F'$ is connected.
Thus, $F'$ is actually 2-connected. Therefore, there is a path
between $v$ and $w$ avoiding $u$, which implies the lemma.  
\end{proof}

The colorless CFI construction has a useful property.

\begin{lemma}\label{lem:vt}
If $F$ is a connected vertex-transitive graph, then both $A$ and $B$ are vertex-transitive.
\end{lemma}

\begin{proof}
We prove that, more generally, any twisted version $A^R$ of $A$ is vertex-transitive.

\begin{claim}\label{cl:vt-a}
For every two vertices $w$ and $w'$ of $F$, there is an automorphism of $A^R$
taking the cell $Q(w)$ onto the cell~$Q(w')$.
\end{claim}

\begin{subproof}
Let $\phi$ be an automorphism of $F$ such that $\phi(w)=w'$.
Fix a permutation $\alpha$ of $V(A)$ satisfying Conditions \refeq{Q}--\refeq{Qbu}.
Let $A^S$ be the image of $A^R$ under $\alpha$.
Since $\alpha$ is an isomorphism from $A^R$ to $A^S$, Lemma \ref{lem:twist1}
implies that $|R\symdiff S|$ is even (we do not need the assumptions on $F$
made in Lemma \ref{lem:twist1} because $\alpha$ maps cells to cells by construction). 
Therefore, Lemma \ref{lem:twist2} applies,
and we have an isomorphism $\beta$ from $A^S$ to $A^R$ mapping each cell onto itself.
The composition $\beta\alpha$ is an automorphism of $A^R$ mapping each cell $Q(z)$
onto the cell $Q(\phi(z))$, in particular, $Q(w)$ onto $Q(w')$.
\end{subproof}

\begin{claim}\label{cl:vt-b}
For every $u\in V(F)$ and every two vertices $x,y\in Q(u)$, there is an automorphism of $A^R$
taking $x$ to $y$ (and vice versa).
\end{claim}

\begin{subproof}
Identifying $Q(w)$ with a subgroup of $(\bZ_2)^k$, let $s=x\oplus y$.
For the permutation $\lambda_s$ in $\Lambda_{Q(u)}$ we have $\lambda_s(x)=y$.
This permutation twists an even number of interspaces incident to the cell $Q(u)$.
Let $A^R[Q(u),Q(v)]$ and $A^R[Q(u),Q(w)]$ be two twisted interspaces.
By Lemma \ref{lem:vt-elem}, the path $vuw$ extends in $F$ to a cycle $vuwz_1\ldots z_kv$.
Using the notation introduced in the proof of Lemma \ref{lem:twist2}, we form the permutation
\begin{equation}
  \label{eq:along-cycle}
\lambda_{vuw}\lambda_{uwz_1}\lambda_{wz_1z_2}\lambda_{z_1z_2z_3}\ldots\lambda_{z_kvu},  
\end{equation}
which is an automorphism of $A^R$.
Split the set of the interspaces twisted by $\lambda_s$ into pairs and
consider the product of automorphisms as in \refeq{along-cycle} for each pair.
The restriction of this automorphism of $A^R$ to the cell $Q(u)$ coincides with $\lambda_s$
and, therefore, it takes $x$ to~$y$.
\end{subproof}

Combining Claims \ref{cl:vt-a} and \ref{cl:vt-b}, we obtain the lemma.
\end{proof}

We now define a graph $G$ as the vertex-disjoint union of two copies
of $A$, and $H$ as the vertex-disjoint union of $A$ and $B$.
Since $A$ is vertex-transitive, $G$ is vertex-transitive as well.
On the other hand, $H$ is not vertex-transitive because 
$A$ and $B$ are connected and non-isomorphic.

The graphs $A$ and $B$ are indistinguishable by \kwl for all $k<s(F)$
by Lemma \ref{lem:sep}, and two copies of the graph $A$ are indistinguishable
by \kwl for every $k$ just because they are isomorphic. 
It follows that $G$ and $H$ are indistinguishable by \kwl for all $k<s(F)$.
This implication can directly be seen from the game characterization
of the \kwl-equivalence relation in \cite{CaiFI92}.
To complete the proof of Theorem \ref{thm:vt}, we therefore need
a family of vertex-transitive graphs $F$ satisfying Assumption \ref{ass:}
and having sufficiently large value of the parameter~$s(F)$.

The \emph{vertex expansion} of a graph $F$ is defined as
$$
h_\mathrm{out}(F)=\min_{0<|S|\le v(F)/2}\frac{|\partial_\mathrm{out}(S)|}{|S|},
$$
where $S\subset V(F)$ and $\partial_\mathrm{out}(S)$ denotes the set of 
vertices of $F$ outside $S$ with at least one neighbor in $S$.
We use the estimate
\begin{equation}
  \label{eq:pvv}
s(F)\ge\frac{h_\mathrm{out}(F)}{3+h_\mathrm{out}(F)}\,v(F)  
\end{equation}
(see Lemma 7.10 in the preliminary version of \cite{PikhurkoVV06}).
Babai \cite{Babai91} estimated the vertex expansion of a connected vertex-transitive graph $F$ 
from below in terms of the diameter of $F$, which we denote by $\diam(F)$. Specifically,
$$
h_\mathrm{out}(F)\ge\frac1{2\diam(F)}.
$$
By \refeq{pvv}, this yields
\begin{equation}
  \label{eq:sD}
s(F)\ge\frac{v(F)}{6\diam(F)+1}.
\end{equation}

To obtain Part 1 of the theorem, consider $F=\cay(D_{2q},\{a,ab,ab^r\})$
for positive integer parameters $q$ and $r$.
Here $D_{2q}$ is the dihedral group with generators $a$ and $b$, where $a$
corresponds to a reflection and $b$ corresponds to a rotation by $2\pi/q$.
Note that, like $a$, the elements $ab$ and $ab^r$ correspond to reflections
and, hence, are involutory. Being a Cayley graph, $F$ is vertex-transitive.

\begin{lemma}\label{lem:diam}
Let $F=\cay(D_{2q},\{a,ab,ab^r\})$.
  \begin{enumerate}[\bf 1.]
  \item 
$\diam(F)\le\frac{2q}r+r+1$.
\item 
If $r<q/2$, then $\nu(F)=1$.
  \end{enumerate}
\end{lemma}

\begin{proof}
\textit{1.}
  The diameter of a Cayley graph $\cay(\Gamma,Z)$ is equal to the minimum $d$
  such that every element of $\Gamma$ is representable as a product
  of at most $d$ elements in $Z$. Every element $b^{sr}$ of $D_{2q}$, where
  $-\lceil q/r\rceil\le s\le\lceil q/r\rceil$, can be represented
  as a product of at most $2\lceil q/r\rceil$ elements $a$ and $ab^r$,
  just because $b^{sr}=(aab^r)^s$.
  To obtain an arbitrary element $b^i$ from the nearest $b^{sr}$,
  it is enough to make at most $2\lfloor(r-1)/2\rfloor$ extra multiplications by $a$ and $ab$.
  An element $ab^i$ is obtainable similarly to $b^{-i}$, with one multiplication
  by $a$ omitted.

\textit{2.} 
The equality $\nu(F)=1$ is obviously equivalent to the condition that $F$ does not contain any subgraph
isomorphic to the 4-cycle graph.
A 4-cycle in $\cay(D_{2q},\{a,ab,ab^r\})$ corresponds to a sequence $c_1,c_2,c_3,c_4$
with $c_i\in\{a,ab,ab^r\}$ such that the product $c_1c_2c_3c_4$ is equal to the identity element of $D_{2q}$.
Moreover, $c_{i+1}\ne c_i$ because all three elements of the connection set are involutory.
A direct inspection shows that no such sequence exists if $r<q/2$.
\end{proof}

Setting $r=\lceil\sqrt{2q}\rceil$ in Lemma \ref{lem:diam}, we obtain a 3-regular vertex-transitive graph $F$
on $2q$ vertices of diameter less than $2\sqrt{2q}+2$. Moreover,
$\nu(F)=1$ if $q\ge11$. Therefore, Assumption \ref{ass:} is fulfilled.
The template graph $F$ yields $G$ and $H$ with $n=16q$ vertices.
By \refeq{sD}, $s(F)>q/(6\sqrt{2q}+7)>0.01\sqrt n$, implying Part~1.

For Part 2 we use 3-regular vertex-transitive graphs with $s(F)=\Omega(v(F))$.
A family of such graphs has been found by Chiu \cite{Chiu92}.
Every graph $F$ in this family is a Cayley graph of
the projective general linear group $\mathrm{PGL}(2,\fp)$ for a prime $p$
such that such that $-2$ and $13$ are quadratic residues modulo $p$;
any prime $p$ such that $p\equiv 1\pmod{104}$ is suitable. 
Every graph in this sequence contains no 4-cycle as its girth is bounded from below
by $2\log_2p$; see \cite[Theorem 5.4]{Chiu92}. Thus, $\nu(F)=1$.
Moreover, every $F$ is a Ramanujan graph.
In general, a $d$-regular graph $F$ is a \emph{Ramanujan graph}
if its second eigenvalue $\lambda_2(F)$ is smaller than or equal to $2\sqrt{d-1}$.
In our case $d=3$ and $\lambda_2(F)\le2\sqrt{2}$.

The \emph{edge expansion} of $F$ is defined as 
$$
h(F)=\min_{0<|S|\le v(F)/2}\frac{|\partial S|}{|S|},
$$
where $S\subset V(F)$ and $\partial S$ denotes the set of edges of $F$ between one vertex
in $S$ and another vertex outside. 
It is known \cite[Theorem 2.4]{HooryLW06} that 
$$
h(F)\ge\frac{d-\lambda_2(F)}2
$$
for a $d$-regular $F$. For $F$ constructed in \cite{Chiu92} we, therefore,
have $h(F)\ge\frac{3-2\sqrt2}2$.

As easily seen, $h_\mathrm{out}(F)\ge h(F)/d$ for $d$-regular graphs, which yields
$h_\mathrm{out}(F)\ge\frac{3-2\sqrt2}6$.
Bound \refeq{pvv} implies in this case that
$s(F)>0.008\,v(F)=0.001\,n$, as desired.
The proof of Theorem \ref{thm:vt} is complete.

\section{Arc-transitivity}\label{s:at}

We begin with a positive result.
Equality \refeq{wl-inv} implies that, if $G$ is arc-transitive,
then the color $\algt{u,v}$ is the same whenever $u$ and $v$ are adjacent.
If this condition is fulfilled, we say that \wl \emph{does not split the adjacency
relation}. The following result reduces recognition of
arc-transitivity of graphs with a prime number of vertices to
verification that Conditions \ref{item:1}--\ref{item:3} listed in
Theorem \ref{thm:p} are met and the adjacency relation is not split. 

\begin{theorem}\label{thm:p-at}
  A vertex-transitive graph $G$ with a prime number of vertices is
  arc-transitive if and only if \wl does not split the adjacency
relation of~$G$. 
\end{theorem}

\begin{proof}
  The necessity part is clear.
  To prove the sufficiency, assume that a vertex-transitive graph $G$ with a
  prime number of vertices is not arc-transitive. This means that the
  adjacency relation of $G$ consists of two or more 2-orbits of
  the automorphism group $\aut(G)$. By Proposition \ref{prop:trans-schur},
  the coherent closure $\clo G$ is schurian. In other words, \wl
  splits the Cartesian square $V(G)^2$ into 2-orbits of $\aut(\clo G)=\aut(G)$
  and, therefore, it splits the adjacency relation of~$G$.
\end{proof}

The method following from Theorem \ref{thm:p-at}, like any other method
based on \kwl with a fixed dimension $k$, cannot be extended
to detecting arc-transitivity on all input graphs.
The following theorem shows that, if Theorem \ref{thm:p-at} admits an extension
to the $n$-vertex input graphs for $n$ in a set $S$, then $S$ can
contain only finitely many numbers of certain form.
We show this by elaborating on the approach presented in Section~\ref{s:vt}.

\begin{theorem}\label{thm:at}\hfill
  \begin{enumerate}[\bf 1.]
  \item 
  For every square number $n$ divisible by $16$, there are $n$-vertex graphs $G$ and $H$
  such that $G$ is arc-transitive, $H$ is not, and $G$ and $H$
  are indistinguishable by \kwl as long as $k\le0.05\,\sqrt n$.
\item 
  For every integer $n=16p$ with a prime factor $p\equiv1\pmod3$,
there are $n$-vertex graphs $G$ and $H$
  such that $G$ is arc-transitive, $H$ is not, and $G$ and $H$
  are indistinguishable by \kwl as long as $k\le0.06\,\sqrt n$.
  \end{enumerate}
\end{theorem}

The proof takes the rest of this section. We need a strengthening of Lemma~\ref{lem:vt-elem}.

\begin{lemma}\label{lem:vt-less-elem}
If $x,y,z$ are three vertices adjacent to a vertex $v$ in a vertex-transitive graph $F$,
then the 3-path $yvz$ extends to a cycle avoiding the vertex~$x$.
\end{lemma}

\begin{proof}
Let $F'$ be the connected component of $F$ containing the vertex $v$.
Note that $F'$ is vertex-transitive. Let $F''$ be obtained from $F'$
by removal of the vertices $v$ and $x$. We use the fact \cite{Watkins70}
that every connected vertex-transitive graph of vertex degree more than 2
is actually 3-connected. Applying it to $F'$, we conclude that $F''$
is connected. Let $P$ be a path between the vertices $y$ and $z$ in $F''$.
Adding the edges $yv$ and $vz$ to $P$, we obtaing a cycle in $F$ that avoids~$x$.
\end{proof}

Let $A$ and $B$ be the graphs constructed in Section~\ref{ss:CFI}.

\begin{lemma}\label{lem:at}
If $F$ is a connected arc-transitive graph, then both $A$ and $B$ are arc-transitive.
\end{lemma}

\begin{proof}
We will prove that any twisted version $A^R$ of $A$ is arc-transitive.
The proof of the following fact is similar to the proof of Claim~\ref{cl:vt-a};
it uses only the arc-transitivity of $F$ and does not need the stronger condition of
2-arc-transitivity.

\begin{claim}\label{cl:at-a}
For every two pairs $vv'$ and $uu'$ of adjacent vertices of $F$, 
there is an automorphism of $A^R$ mapping $Q(v)$ onto $Q(u)$ and $Q(v')$ onto $Q(u')$.
\end{claim}

Claim \ref{cl:at-a} reduces proving that $A^R$ is arc-transitive to proving the following fact.

\begin{claim}\label{cl:at-b}
Let $vv'\in E(F)$. For every two pairs of adjacent vertices $xx'$ and $yy'$ of the interspace $A^R[Q(v),Q(u)]$
such that $x,y\in Q(v)$ and $x',y'\in Q(v')$, 
there is an automorphism of $A^R$ taking $x$ to $y$ and $x'$ to~$y'$.
\end{claim}

\begin{subproof}
Recall that $A[Q(u),Q(v)]\cong 2K_{2^{k-2},2^{k-2}}$.
Notice first that there is an automorphism $\alpha$ of $A^R$ transposing the two $K_{2^{k-2},2^{k-2}}$-parts.
Indeed, let $C$ be a cycle in $F$ containing the edge $uv$, which exists by Lemma \ref{lem:vt-elem},
and consider the permutation $\alpha$, constructed similarly to the permutation \refeq{along-cycle} in the 
proof of Claim \ref{cl:vt-b}, that twists the subgraph $A[Q(x),Q(y)]$ twice for each edge $xy$ along~$C$.

It remains to show that $A^R$ has an automorphism $\alpha$ transposing $x$ and $y$
and fixing each of $x'$ and $y'$ (the argument works as well for the symmetric case
when $x$ and $y$ are fixed while $x'$ and $y'$ are transposed).
Like in the proof of Claim \ref{cl:vt-b}, consider the permutation $\lambda_s$ in $\Lambda_{Q(v)}$
for $s=x\oplus y$ and argue that there exists an $\alpha\in\aut(A^R)$ whose restriction to the cell $Q(v)$ 
coincides with $\lambda_s$ and whose restriction to $Q(v')$ is the identity. 
We split the set of the interspaces twisted by $\lambda_s$ into pairs,
and for each pair form an automorphism of $A^R$ similarly to \refeq{along-cycle}
along a cycle in $F$ avoiding the vertex $v'$. Such a cycle exists by Lemma \ref{lem:vt-less-elem}.
Let $\alpha$ be the product of these automorphisms. It remains to note that
$\alpha$ induces $\lambda_s$ on $Q(v)$ and does not touch any vertex in $Q(v')$, as desired.
\end{subproof}

The proof of the lemma is complete.
\end{proof}

In order to prove Part 1 of Theorem \ref{thm:at},
we specify $F$ by taking $F=\cay(\bZ_\ell\times\bZ_\ell,\Set{\pm(1,0),\pm(0,1)})$, where $\ell\ge3$.
This graph fulfills Assumption \ref{ass:} because it is 4-regular and $\nu(F)=2$.
Like in the proof of Theorem \ref{thm:vt}, let $G$ be the vertex-disjoint union
of two copies of $A$, and $H$ be the vertex-disjoint union
of $A$ and $B$. Since these are connected non-isomorphic graphs, $H$ is 
not arc-transitive (in fact, not even vertex-transitive).

Being a Cayley graph, $F$ is vertex-transitive. Taking into account its automorphisms
$(x,y)\mapsto(y,x)$ and $(x,y)\mapsto(x,-y)$, we see that $F$ is arc-transitive.
By Lemma \ref{lem:at}, $A$ and, therefore, $G$ are arc-transitive.

The diameter of $F$ does not exceed $\ell$.
Bound \refeq{sD} implies that $s(F)\ge\frac{\ell^2}{6\ell+1}\ge\frac16\,\ell-1$.
By Lemma \ref{lem:sep}, $G$ and $H$ are indistinguishable by \kwl for all $k<\frac16\,\ell=0.05\sqrt n$,
where $n=16\ell^2$ is the number of vertices in $G$ and in $H$.

For Part 2, we use a result by Cheng and Oxley \cite{ChengO87} who proved that, for every prime $p$ with $p\equiv1\pmod3$,
up to isomorphism there exists exactly one arc-transitive 3-regular graph with $2p$ vertices.
Denote this graph by $F_p$. It is known (e.g., \cite{FengK02})
that $F_p$ is isomorphic to a Cayley graph of the dihedral group $D_{2p}$.
In fact, we can set $F_p=\cay(D_{2p},\{ab,ab^r,ab^{r^2}\})$,
where $r\ne1$ is a cube root of unity in the multiplicative group $\fpm$.
In order to show that this graph is arc-transitive, it suffices to check
that $\aut(F_p)$ acts transitively on the neighborhood of the identity element $e$ of $D_{2p}$
in $F_p$, that is, on the set $N(e)=\{ab,ab^r,ab^{r^2}\}$. Define a map
$\alpha_s\function{D_{2p}}{D_{2p}}$ by $\alpha_s(b^i)=b^{si}$ and $\alpha_s(ab^i)=ab^{si}$. 
If $s$ is coprime to $p$, then $\alpha_s$ is an automorphism of $D_{2p}$
and, therefore, it is an automorphism of $F_p$.
It remains to note that $\{\alpha_1,\alpha_r,\alpha_{r^2}\}$
is a subgroup of $\aut(F_p)$ acting transitively on~$N(e)$.

The graph $F_p$ fulfills Assumption \ref{ass:} because $\nu(F_p)=1$.\footnote{In fact,
$F_p$ is of girth 6, where a shortest cycle corresponds to the equality $(abab^rab^{r^2})^2=1$.}
We now estimate the diameter of~$F_p$.

\begin{lemma}\label{lem:diam-2}
$\diam(F_p)<4\sqrt{3p}+7$.
\end{lemma}

\begin{proof}
We begin with a useful observation that, since $r$ is a root of the polynomial $x^3-1=(x-1)(x^2+x+1)$
in $\fp$, we have
\begin{equation}
  \label{eq:rrr}
r^2+r+1=0  
\end{equation}
in this field.

Define a map
$\beta_t\function{D_{2p}}{D_{2p}}$ by $\beta_t(b^i)=b^{i}$ and $\alpha_s(ab^i)=ab^{i+t}$. 
This is an automorphism of $D_{2p}$ for every integer $t$. The composition $\gamma=\beta_1\circ\alpha_{r-1}$,
where $\alpha_s$ is as defined above, is an automorphism of $D_{2p}$.
Note that $\gamma(a)=ab$, $\gamma(ab)=ab^r$, and $\gamma(ab^{r+1})=ab^{r^2}$.
It follows that $F_p$ is isomorphic to $\cay(D_{2p},\{a,ab,ab^{r+1}\})$,
and below we will work with this representation of~$F_p$.

Given a group $\Gamma$ and a set $Z$ of its elements, let
$\rho(\Gamma,Z)$ denote the minimum $d$
  such that every element of $\Gamma$ is representable as a product
  of at most $d$ elements in $Z$. As was already noted in the proof of Lemma \ref{lem:diam},
in the case that $Z$ is the connection set of a Cayley graph of $\Gamma$,
we have
\begin{equation}
  \label{eq:diam-rho}
\diam(\cay(\Gamma,Z))=\rho(\Gamma,Z).  
\end{equation}
Note that
\begin{equation}
  \label{eq:ababa}
aab=b,\ abab^{r+1}=b^r,\text{ and }aab^{r+1}=b^{r+1}=b^{-r^2},  
\end{equation}
where the last equality follows from \refeq{rrr}.
Looking at the cyclic subgroup of $D_{2p}$ generated by $b$, we see that
\begin{equation}
  \label{eq:rho-rho}
\rho(D_{2p},\{a,ab,ab^{r+1}\})\le1+2\,\rho(\fpa,\{1,r,-r^2\}),  
\end{equation}
where the first summand corresponds to the multiplication by $a$ in $D_{2p}$,
and the factor of 2 in the second summand is due to Equalities \refeq{ababa}.
Our goal is, therefore, to estimate $\rho(\fpa,\{1,r,-r^2\})$.

Given two sets $A$ and $B$ of elements of an abelian group,
we use the standard notation $A+B=\Set{a+b}{a\in A,\,b\in B}$
and also inductively define $1\,A=A$ and $m\,A=(m-1)\,A+A$ for
an integer $m>1$.

Set
$$
m=\lceil\sqrt{p/3}\rceil.
$$
Sergei Konyagin communicated to us a proof that, if $A=\{1,r,r^2\}$
is the order-3 subgroup of the multiplicative group $\fpm$, then
the set $2m\,A$ contains a fraction of the elements of $\fpa$ independent of~$p$.
We use his argument to prove the following similar fact for~$\fpa$.

\begin{claim}\label{cl:Konyagin}
$|2m\,\{1,r,-r^2\}|\ge m^2$.  
\end{claim}

\begin{subproof}
It suffices too show that all elements $a_1r-a_2r^2$ of $\fpa$ for $1\le a_1,a_2\le m$
are pairwise distinct. Assume towards a contradiction that this is not the case, that is,
$$
a_1r-a_2r^2=a'_1r-a'_2r^2
$$
for $(a_1,a_2)\ne(a'_1,a'_2)$; here and below all equalities are supposed to be in $\fp$. This implies
$$
b_1r = b_2r^2
$$
for some integers $b_1$ and $b_2$ not both equal to zero and such that $|b_1| < m$,
$|b_2| < m$. Then, in fact, neither $b_1$ nor $b_2$ is equal to zero.
Substituting $r=b_1/b_2$ in \refeq{rrr}, we obtain
$$
(b_1/b_2)^2 + (b_1/b_2) +1 =0.
$$
For the integers $b_1$ and $b_2$ this means that
\begin{equation}
  \label{eq:bbbb}
b_1^2+b_1b_2+b_2^2\equiv 0\pmod p.  
\end{equation}
On the other hand, $b_1^2+b_1b_2+b_2^2\ne0$ because
$$
b_1^2+b_1b_2+b_2^2>b_1^2-2|b_1||b_2|+b_2^2=(|b_1|-|b_2|)^2.
$$
Moreover, $b_1$ and $b_2$ satisfy the bounds
$$-p<-m^2<b_1^2+b_1b_2+b_2^2\le 3(m-1)^2<p,$$
and we have a contradiction with~\refeq{bbbb}.
\end{subproof}

\begin{claim}\label{cl:CD}
$6m\,\{1,r,-r^2\}=\fpa$.  
\end{claim}

\begin{subproof}
The Cauchy–Davenport theorem, a  classical result in additive combinatorics, says that, for sets $A,B\subset\fpa$,
$$
|A+B|\ge\min(|A|+|B|-1,p).
$$
By Claim \ref{cl:Konyagin}, 
$$
|2m\,\{1,r,-r^2\}|\ge m^2 > p/3. 
$$
Since $p\equiv1\pmod3$, we actually have
$$
|2m\,\{1,r,-r^2\}|\ge(p+2)/3. 
$$
Applying the Cauchy–Davenport theorem for $A=B=2m\,\{1,r,-r^2\}$
and once again for $A=2m\,\{1,r,-r^2\}$ and $B=4m\,\{1,r,-r^2\}$, we obtain the estimate
$|6m\,\{1,r,-r^2\}|\ge p$, implying the claim.
\end{subproof}

Claim \ref{cl:CD} can be recast as the estimate
$$
\rho(\fpa,\{1,r,-r^2\})\le6m<6(\sqrt{p/3}+1).
$$
The required bound for $\diam(F_p)$ now follows from
Equality \refeq{diam-rho} and Inequality~\refeq{rho-rho}.
\end{proof}

Using Lemma \ref{lem:diam-2} and Bound \refeq{sD}, we obtain
$$
s(F_p)\ge\frac{2p}{6\diam(F_p)+1}>\frac{\sqrt p}4
$$
for all $p\ge43$. By Lemma \ref{lem:sep}, the template graph $F_p$ results in two graphs $G$ and $H$
that are indistinguishable by \kwl as long as $k\le\sqrt p/4=\sqrt n/16$,
where $n=16p$ is the number of vertices in $G$ and in $H$.
If $p<43$, this bound for the WL dimension is also true, just saying that $G$ and $H$ are indistinguishable by~\WL1.

The proof of Theorem \ref{thm:at} is complete.

\section{The Weisfeiler-Leman regularity hierarchy}\label{s:hierarchy}

The stabilized \WL k-coloring of $k$-tuples of vertices determines 
a canonical coloring of $s$-tuples for each $s$ between 1 and $k$.
Specifically, if $s<k$, we define $\kalgt{x_1,\ldots,x_s}=\kalgt{x_1,\ldots,x_s,\ldots,x_s}$
just by cloning the last vertex in the $s$-tuple $k-s$ times.

\begin{definition}\label{def:homog}
Let $s\le k$.
A graph $G$ is called \emph{$\wlh sk$-regular}
if \WL k does not make any
non-trivial splitting of $V(G)^s$, that is, $\kalgt\barx\ne\kalgt\bary$
exactly when $\alg k0\barx\ne\alg k0\bary$
for every pair of $s$-tuples $\barx,\bary\in V(G)^s$.
\end{definition}

Note that a graph is $\wlh 11$-regular if and only if it is regular,
and it is $\wlh 22$-regular if and only if it is strongly regular.
The class of $\wlh 12$-regular graphs does not seem to have that clear characterization.
This class obviously contains all vertex-transitive graphs. Moreover,
it contains all constituent graphs of association schemes
(i.e., symmetric closures of basis relations in association schemes).
The last class, in turn, contains
strongly regular and distance-regular graphs \cite{BrouwerCN89}
(in fact, every connected strongly regular graph is distance-regular).
If two $\wlh 12$-regular graphs $G_1$ and $G_2$ are indistinguishable by \WL2,
then the disjoint union of $G_1$ and $G_2$ is also $\wlh 12$-regular.
Furthermore, the class is closed under graph complementation.
Note also that a graph $G$ is $\wlh 12$-regular if and only if
the coherent closure of $G$ is an association scheme.

Denote the class of all $\wlh sk$-regular graphs by $\wlhi sk$.
We first note that the $\wlhi sk$ hierarchy collapses
as the parameter $s$ increases. To show this, we relate it
to two known graph-theoretic symmetry and regularity concepts.

A graph $G$ is \emph{$s$-ultrahomogeneous} if every isomorphism
between two induced subgraphs of $G$ with at most $s$ vertices extends
to an automorphism of the whole graph $G$. 
Denote the class of all $s$-ultrahomogeneous graphs by $\uhi s$.
Note that 1-ultrahomogeneous
graphs are exactly vertex-transitive graphs, and 2-ultrahomoge\-neous graphs
are rank 3 graphs. Cameron \cite{Cameron80} proved that
every 5-ultrahomogeneous graph is $s$-ultrahomogeneous for all $s\ge5$, i.e.,
$\uhi s=\uhi5$ for $s\ge5$.
All graphs in $\uhi5$ were identified by Gardiner \cite{Gardiner78}.
Those are $mK_n$ (the vertex-disjoint union of $m$ copies of the complete graph $K_n$),
their complements $\overline{mK_n}$ (i.e., the regular multipartite graphs),
the 5-cycle graph $C_5$, and the $3\times3$-rook's graph (or, the same,
the line graph $L(K_{3,3})$ of the complete bipartite graph $K_{3,3}$).

A graph $G$ is called \emph{$s$-tuple regular} (or, sometimes, \emph{$s$-isoregular}) 
if the number of common neighbors
of any set $S$ of at most $s$ vertices in $G$ depends only on the isomorphism
type of the induced subgraph $G[S]$. 
Let us denote the class of all $s$-tuple regular graphs by $\tri s$.
Note that
1-tuple regular graphs are exactly regular graphs, and 2-tuple regular graphs
are exactly strongly regular graphs. Cameron \cite{Cameron80} (see also \cite[Theorem 8.21]{CameronvL91})
and, independently, Gol'fand (see the historical comments in \cite{Babai95,CaiFI92}) proved
that $\tri5=\uhi5$. As a consequence, $\tri s=\tri5$ for all~$s\ge5$.

Let $s\le k$. It is clear that 
$$
\uhi s\subseteq\wlhi sk\subseteq\tri s.
$$
It follows that
$$
\wlhi sk=\Set{mK_n,\overline{mK_n}}_{m,n}\cup\Set{C_5,L(K_{3,3})}\text{\ \ if\ \ }k\ge s\ge5.
$$

Let us fix the first parameter to $s=1$. 
Since the $\wlhi sk$ hierarchy collapses with respect to the parameter $k$ for each $s\ge5$,
it is natural to ask whether the hierarchy of the classes $\wlhi 1k$
collapses with respect to $k$. More precisely, note that
\begin{equation}\label{eq:wlhi1}
\wlhi11\supseteq\wlhi12\supseteq\wlhi13\supseteq\ldots\supseteq\vtclass,
\end{equation}
where $\vtclass=\uhi1$ denotes the class of all vertex-transitive graphs.
We say that the hierarchy \refeq{wlhi1} \emph{collapses at level $K$}
if $\wlhi1k=\wlhi1K$ for all $k\ge K$. Theorem \ref{thm:vt} has the following
consequence.

\begin{corollary}
The hierarchy \refeq{wlhi1} does not collapse at any level.
\end{corollary}

\begin{proof}
  Note that the equality $\vtclass=\wlhi 1K$ would mean that
  vertex-transitivity were recognizable just by running \WL K on $G$ and looking whether
the resulting partition of $V(G)$ is non-trivial.
Theorem \ref{thm:vt}, therefore, implies that $\vtclass\ne\wlhi 1K$
for any $K$. Let $G$ be a non-vertex-transitive graph in $\wlhi 1K$.
Let $k$ denote the number of vertices in $G$.
Since $G$ obviously does not belong to $\wlhi1k$, we
conclude that $\wlhi1k$ is a proper subclass of~$\wlhi1K$.
\end{proof}

\section{Concluding discussion}

We have suggested a new, very simple combinatorial algorithm
recognizing, in polynomial time, vertex-transitivity of
graphs with a prime number of vertices.
The algorithm consists, in substance, in running \wl
on an input graph and all its vertex-individualized copies.
This is perhaps the first example of using the Weisfeiler-Leman algorithm
for recognition of a non-trivial graph property rather than for
isomorphism testing.

One can consider another, conceptually even simpler approach.
If an input graph $G$ is vertex-transitive,
then \kwl colors all diagonal $k$-tuples $(u,\ldots,u)$, $u\in V(G)$,
in the same color. Is this condition for a possibly large, but fixed $k$
sufficient to claim vertex-transitivity? In general, a negative
answer immediately follows from Theorem \ref{thm:vt}.
Does there exist a fixed dimension $k$ such that the answer
is affirmative for graphs with a prime number of vertices?
This is apparently a hard question; it seems that we cannot even exclude
that $k=3$ suffices.

Another interesting question is whether \kwl is able to efficiently
recognize vertex-transitivity on $n$-vertex input graphs for $n$
in a larger range than the set of primes. The lower bound of
Theorem \ref{thm:vt} excludes this only for $n$ divisible by 16,
in particular, for the range of $n$ of the form $16p$ for a prime $p$.
Can \kwl be successful on the inputs with $2p$ vertices?
Conversely, can the negative result of Theorem \ref{thm:vt}
be extended to a larger range of~$n$?\footnote{Note that Theorem \ref{thm:vt}
can be extended in a straightforward way to the range of $n$ divisible by $8p$
for each prime $p$. It is enough to define $G$ as the vertex-dijoint union of $p$
copies of $A$ and $H$ as the vertex-dijoint union of $p-1$
copies of $A$ and one copy of~$B$.}
These questions also make a perfect sense for arc-transitivity.

We conclude with a brief discussion of
the concept of a \emph{rank 3 graph}, which is a ``2-dimensional'' analog
of vertex-transitivity. While the automorphism group
of a vertex-transitive graph acts transitively on its vertex set,
the automorphism group of a rank 3 graph acts transitively not only
on the vertices, but also on the ordered pairs of adjacent vertices
and on the ordered pairs of non-adjacent vertices.
Thus, $G$ is a rank 3 graph if and only if both $G$ and its complement
are arc-transitive. Since this concept is more stringent
than vertex- and arc-transitivity, the corresponding
recognition problem has a priori better chances to be solvable by \kwl,
though its complexity status is wide open.

Let $P$ be a pattern graph with designated vertices $u_1$ and $u_2$.
For vertices $x_1$ and $x_2$ of a graph $G$, let $s_P(x_1,x_2)$
denote the number of copies of $P$ in $G$ such that $u_1$
is placed at $x_1$ and $u_2$ is placed at $x_2$.
A graph $G$ satisfies the \emph{$t$-vertex condition}
if, for every $P$ with at most $t$ vertices, the count $s_P(x_1,x_2)$
depends only on (non)equality and (non)adjacency of $x_1$ and $x_2$
(Higman \cite{Higman70}).
If $G$ is a rank 3 graph, then the $t$-vertex condition
is obviously satisfied for every $t$.
The long-standing Klin's conjecture \cite{FaradzevKM94}
(see also the references in \cite{Reichard00,Pech20}) 
says that there is an integer $T$ such that, conversely, 
every graph $G$ satisfying the $T$-vertex condition is a rank 3 graph.
We remark that, if this conjecture is true, then \WL{T}
recognizes the class of 3 rank graphs in a straightforward way.
This follows from a simple observation that, if
the color of a $T$-tuple $(x_1,x_2,\ldots,x_2)$ in $\mathrm{WL}_T(G)$
depends only on the type of the pair $(x_1,x_2)$ (like
for rank 3 graphs), then $G$ satisfies the $T$-vertex condition.

\subsection*{Acknowledgement}
We gratefully acknowledge a crucial contribution of Sergei Konyagin in the proof of Lemma~\ref{lem:diam-2}.


\begin{thebibliography}{10}

\bibitem{Alspach73}
B.~Alspach.
\newblock Point-symmetric graphs and digraphs of prime order and transitive
  permutation groups of prime degree.
\newblock {\em J. Comb. Theory, Ser. B}, 15:12--17, 1973.

\bibitem{Babai77}
L.~Babai.
\newblock Isomorphism problem for a class of point-symmetric structures.
\newblock {\em Acta Math. Acad. Sci. Hungar.}, 29(3-4):329--336, 1977.

\bibitem{Babai80}
L.~Babai.
\newblock On the complexity of canonical labeling of strongly regular graphs.
\newblock {\em {SIAM} J. Comput.}, 9(1):212--216, 1980.

\bibitem{Babai91}
L.~Babai.
\newblock Local expansion of vertex-transitive graphs and random generation in
  finite groups.
\newblock In {\em Proceedings of the 23rd Annual {ACM} Symposium on Theory of
  Computing (STOC'91)}, pages 164--174. ACM, 1991.

\bibitem{Babai95}
L.~Babai.
\newblock Automorphism groups, isomorphism, reconstruction.
\newblock In {\em Handbook of Combinatorics}, chapter~27, pages 1447--1540.
  Elsevier, 1995.

\bibitem{Babai16}
L.~Babai.
\newblock Graph isomorphism in quasipolynomial time.
\newblock In {\em Proceedings of the 48th Annual {ACM} Symposium on Theory of
  Computing ({STOC}'16)}, pages 684--697, 2016.

\bibitem{BabaiCSTW13}
L.~Babai, X.~Chen, X.~Sun, S.~Teng, and J.~Wilmes.
\newblock Faster canonical forms for strongly regular graphs.
\newblock In {\em Proc. of the 54th Symposium on Foundations of Computer
  Science (FOCS'13)}, pages 157--166. {IEEE} Computer Society, 2013.

\bibitem{Foster}
I.~Bouwer, W.~Chernoff, B.~Monson, and Z.~Star, editors.
\newblock {\em The Foster census. R. M. Foster's census of connected symmetric
  trivalent graphs}.
\newblock Winnipeg: Charles Babbage Research Centre, 1988.

\bibitem{BrouwerCN89}
A.~E. Brouwer, A.~M. Cohen, and A.~Neumaier.
\newblock {\em Distance-regular graphs}, volume~18 of {\em Ergebnisse der
  Mathematik und ihrer Grenzgebiete. 3. Folge}.
\newblock Berlin etc.: Springer-Verlag, 1989.

\bibitem{BrouwerH12}
A.~E. Brouwer and W.~H. Haemers.
\newblock {\em Spectra of graphs}.
\newblock Berlin: Springer, 2012.

\bibitem{CaiFI92}
J.~Cai, M.~F{\"u}rer, and N.~Immerman.
\newblock An optimal lower bound on the number of variables for graph
  identifications.
\newblock {\em Combinatorica}, 12(4):389--410, 1992.

\bibitem{CameronvL91}
P.~Cameron and J.~{van Lint}.
\newblock {\em Designs, graphs, codes and their links}, volume~22 of {\em
  London Mathematical Society Student Texts}.
\newblock Cambridge etc.: Cambridge University Press, 1991.

\bibitem{Cameron80}
P.~J. Cameron.
\newblock 6-transitive graphs.
\newblock {\em J. Comb. Theory, Ser. {B}}, 28(2):168--179, 1980.

\bibitem{CP2019}
G.~Chen and I.~Ponomarenko.
\newblock {\em Coherent configurations}.
\newblock Wuhan: Central China Normal University Press, 2019.
\newblock A draft version is available at
  \url{http://www.pdmi.ras.ru/~inp/ccNOTES.pdf}.

\bibitem{ChengO87}
Y.~Cheng and J.~G. Oxley.
\newblock On weakly symmetric graphs of order twice a prime.
\newblock {\em J. Comb. Theory, Ser. {B}}, 42(2):196--211, 1987.

\bibitem{Chiu92}
P.~Chiu.
\newblock Cubic {R}amanujan graphs.
\newblock {\em Combinatorica}, 12(3):275--285, 1992.

\bibitem{Evdokimov04}
S.~Evdokimov.
\newblock {\em Schurity and separability of association schemes}.
\newblock PhD thesis, St. Petersburg University, St. Petersburg, 2004.

\bibitem{EvdokimovP99}
S.~Evdokimov and I.~Ponomarenko.
\newblock On highly closed cellular algebras and highly closed isomorphisms.
\newblock {\em Electr. J. Comb.}, 6, 1999.

\bibitem{EvdokimovP04}
S.~Evdokimov and I.~Ponomarenko.
\newblock Circulant graphs: recognizing and isomorphism testing in polynomial
  time.
\newblock {\em St. Petersbg. Math. J.}, 15(6):813--835, 2004.

\bibitem{EvdokimovPT00}
S.~Evdokimov, I.~Ponomarenko, and G.~Tinhofer.
\newblock Forestal algebras and algebraic forests (on a new class of weakly
  compact graphs).
\newblock {\em Discrete Mathematics}, 225(1-3):149--172, 2000.

\bibitem{FaradzevKM94}
I.~Farad\v{z}ev, M.~Klin, and M.~Muzichuk.
\newblock Cellular rings and groups of automorphisms of graphs.
\newblock In {\em Investigations in algebraic theory of combinatorial objects},
  pages 1--152. Dordrecht: Kluwer Academic Publishers, 1994.

\bibitem{FengK02}
Y.~Feng and J.~H. Kwak.
\newblock Constructing an infinite family of cubic 1-regular graphs.
\newblock {\em Eur. J. Comb.}, 23(5):559--565, 2002.

\bibitem{stacs20}
F.~Fuhlbr{\"u}ck, J.~K{\"o}bler, and O.~Verbitsky.
\newblock Identifiability of graphs with small color classes by the
  {W}eisfeiler-{L}eman algorithm.
\newblock In {\em 37th Symposium on Theoretical Aspects of Computer Science
  {(STACS'20)}}, volume 154 of {\em LIPIcs}. Schloss Dagstuhl --
  Leibniz-Zentrum f{\"u}r Informatik, 2020.

\bibitem{Fuerer17}
M.~F{\"u}rer.
\newblock On the combinatorial power of the {W}eisfeiler-{L}ehman algorithm.
\newblock In {\em Algorithms and Complexity --- 10th International Conference
  (CIAC'17) Proceedings}, volume 10236 of {\em Lecture Notes in Computer
  Science}, pages 260--271, 2017.

\bibitem{Gardiner78}
A.~Gardiner.
\newblock Homogeneity conditions in graphs.
\newblock {\em J. Comb. Theory, Ser. {B}}, 24(3):301--310, 1978.

\bibitem{Grohe12}
M.~Grohe.
\newblock Fixed-point definability and polynomial time on graphs with excluded
  minors.
\newblock {\em J. ACM}, 59(5):27:1--27:64, 2012.

\bibitem{Higman70}
D.~Higman.
\newblock Characterization of families of rank 3 permutation groups by the
  subdegrees. ii.
\newblock {\em Arch. Math.}, 21:353--361, 1970.

\bibitem{Higman71}
D.~Higman.
\newblock Coherent configurations. i.
\newblock {\em Rend. Semin. Mat. Univ. Padova}, 44:1--25, 1971.

\bibitem{HooryLW06}
S.~Hoory, N.~Linial, and A.~Widgerson.
\newblock Expander graphs and their applications.
\newblock {\em Bulletin of the American Mathematical Society. New Series},
  43(4):439--561, 2006.

\bibitem{dcbook}
N.~Immerman.
\newblock {\em Descriptive complexity}.
\newblock Graduate texts in Computer Science. Springer, 1999.

\bibitem{ImmermanL90}
N.~Immerman and E.~Lander.
\newblock Describing graphs: A first-order approach to graph canonization.
\newblock In {\em Complexity Theory Retrospective}, pages 59--81. Springer,
  1990.

\bibitem{KieferPS19}
S.~Kiefer, I.~Ponomarenko, and P.~Schweitzer.
\newblock The {W}eisfeiler-{L}eman dimension of planar graphs is at most 3.
\newblock {\em J. ACM}, 66(6):44:1--44:31, 2019.

\bibitem{MP2012}
M.~Muzychuk and I.~Ponomarenko.
\newblock On pseudocyclic association schemes.
\newblock {\em Ars Math. Contemp.}, 5(1):1--25, 2012.

\bibitem{MuzychukKP99}
M.~E. Muzychuk, M.~H. Klin, and R.~P{\"{o}}schel.
\newblock The isomorphism problem for circulant graphs via {S}chur ring theory.
\newblock In {\em Codes and Association Schemes, Proceedings of a {DIMACS}
  Workshop, Piscataway, New Jersey, USA, November 9-12, 1999}, volume~56 of
  {\em {DIMACS} Series in Discrete Mathematics and Theoretical Computer
  Science}, pages 241--264. {DIMACS/AMS}, 1999.

\bibitem{MuzychukT98}
M.~E. Muzychuk and G.~Tinhofer.
\newblock Recognizing circulant graphs of prime order in polynomial time.
\newblock {\em Electr. J. Comb.}, 5, 1998.

\bibitem{NeuenS17}
D.~Neuen and P.~Schweitzer.
\newblock Benchmark graphs for practical graph isomorphism.
\newblock In {\em 25th Annual European Symposium on Algorithms (ESA'17)},
  volume~87 of {\em LIPIcs}, pages 60:1--60:14. Schloss Dagstuhl --
  Leibniz-Zentrum für Informatik, 2017.

\bibitem{Pech20}
C.~Pech.
\newblock On highly regular strongly regular graphs.
\newblock Technical report, \url{arxiv.org/abs/1404.7716v5}, 2020.

\bibitem{PikhurkoVV06}
O.~Pikhurko, H.~Veith, and O.~Verbitsky.
\newblock The first order definability of graphs: {U}pper bounds for quantifier
  depth.
\newblock {\em Discrete Applied Mathematics}, 154(17):2511--2529, 2006.
\newblock A preliminary version is available at
  \url{arxiv.org/abs/math/0311041}.

\bibitem{Reichard00}
S.~Reichard.
\newblock A criterion for the {$t$}-vertex condition of graphs.
\newblock {\em J. Comb. Theory, Ser. {A}}, 90(2):304--314, 2000.

\bibitem{Turner67}
J.~Turner.
\newblock Point-symmetric graphs with a prime number of points.
\newblock {\em J. Comb. Theory}, 3:136--145, 1967.

\bibitem{Watkins70}
M.~E. Watkins.
\newblock Connectivity of transitive graphs.
\newblock {\em Journal of Combinatorial Theory}, 8(1):23 -- 29, 1970.

\bibitem{WLe68}
B.~Weisfeiler and A.~Leman.
\newblock The reduction of a graph to canonical form and the algebra which
  appears therein.
\newblock {\em NTI, Ser.~2}, 9:12--16, 1968.
\newblock English translation is available at
  \url{https://www.iti.zcu.cz/wl2018/pdf/wl_paper_translation.pdf}.

\end{thebibliography}
\end{document}